\documentclass[12pt,twoside]{amsart}
\usepackage{amssymb, amsmath}

\usepackage{hyperref}

\hypersetup{
colorlinks=true, 
breaklinks=true, 
urlcolor= blue, 
linkcolor= blue, 
bookmarksopen=true, 
pdftitle={}, 
pdfauthor={St\'ephane Junca}, 
pdfsubject={} 
}

\newcommand{\dis}{\displaystyle}

\newcommand{\eps}{\varepsilon}

\newcommand{\ph}{\varphi}
\newcommand{\alali}{$\mbox{}$\\}

\newcommand{\bqr}{\begin{eqnarray}}
\newcommand{\bqre}{\begin{eqnarray*}}
\newcommand{\eqr}{\end{eqnarray}}
\newcommand{\eqre}{\end{eqnarray*}}

\newcommand{\ue}{u_\eps}
\newcommand{\Ue}{U_\eps}
\newcommand{\ve}{v_\eps}
\newcommand{\Ve}{V_\eps}

\newcommand{\We}{W_\eps}
\newcommand{\Be}{B_\eps}
\newcommand{\Ne}{N_\eps}

 \newcommand{\F}{{\bf F}}  
 \newcommand{\G}{{\bf G}}

  \newcommand{\y}{{\bf y}}
  \newcommand{\vv}{{\bf v}}

\newcommand{\h}{{\bf h}}

\newcommand{\bpro}{  { \rm \bf Proof : }}
\newcommand{\epro} {\leavevmode
  \hbox to.77778em{%
  \hfil\vrule
  \vbox to.675em{\hrule width.6em\vfil\hrule}%
  \vrule\hfil} \\}

\newcommand{\R}{\mathbb{R}}

\newcommand{\Z}{\mathbb{Z}}
\newcommand{\N}{\mathbb{N}}

\newcommand{\del}{{\partial}}

\renewcommand{\a}{{\bf a} } 
\renewcommand{\t}{\theta}

\newcommand{\g}{\gamma}

\newcommand{\M}{\mathcal M}

\renewcommand{\div}{\text{\rm div}\,}

\def\al{ \alpha}

\def\eps{\varepsilon}

\def\div{\mbox{div}}

\def\x{{\bf x}}
\def\F{{\bf F}}
\def\G{{\bf G}}

\def\inf{\hbox{\rm inf}}
\def\sup{\hbox{\rm sup}}

\def\M1{{\cal M}_1}

\theoremstyle{plain}

\newtheorem{proposition}{Proposition}[section]
\newtheorem{thm}{Theorem}[section]
\newtheorem{lem}{Lemma}[section]

\newtheorem{dfni}{Definition}[section]

\theoremstyle{definition}
\theoremstyle{remark}

\newtheorem{remark}{Remark}[section]
\numberwithin{equation}{section}

\begin{document}

\title[  High frequencies and   regularity  for  conservation laws
        ]
{    High frequency waves  and
    \\    the maximal  smoothing  effect 
\\ for nonlinear scalar conservation laws }
   
 \author{St\'{e}phane Junca }
  \date{\today}

 \address[St\'{e}phane Junca]{
 Laboratoire J.A. Dieudonn\'e
 Universit\'{e} de Nice Spohia-Antipolis, UMR CNRS  7351,Parc Valrose, 
 06108, Nice, France.
 }
\address[ St\'{e}phane Junca ]{
Team COFFEE, INRIA Sophia-Antipolis  M\'editerran\'ee,  2004  Routes des Lucioles -BP 93,  06902 Sophia-Antipolis, France}
 \email{junca@unice.fr}


\bigskip
\begin{abstract}
The article first studies  the propagation of well prepared  high frequency waves with small amplitude $\eps$  near constant solutions for entropy solutions of multidimensional nonlinear scalar conservation laws. 
    Second,  such oscillating solutions  are used to highlight a conjecture of Lions, Perthame, Tadmor,  (\cite{LPT}), 
  about the maximal regularizing effect for nonlinear conservation laws. 
For this purpose,   a new   definition  of   smooth nonlinear   flux
 is stated  and compared to classical definitions.   
 Then    it is proved  that the uniform smoothness expected by \cite{LPT} in Sobolev spaces  cannot be  exceeded for all smooth nonlinear fluxes.
\end{abstract}
 \maketitle

{\small
\noindent{\bf Key-words}:  multidimensional conservation laws, nonlinear smooth flux, 
  geometric optics,
 Sobolev spaces, smoothing effect.
\medskip

\noindent {\bf Mathematics Subject Classification}:
\\
   Primary: 35L65, 35B65; Secondary: 35B10, 35B40,  35C20.
}

\tableofcontents

\section{Introduction}\label{sI}
\smallskip
 
 This paper deals with super critical  geometric optics to highlight the maximal regularizing effect for nonlinear multidimensional scalar
conservation laws. This   effect  is   studied  in Sobolev spaces  by P.L. Lions,  B. Perthame and  E. Tadmor in  \cite{LPT}. 
They obtain an uniform  fractional Sobolev bounds for any ball of $L^\infty$ initial data under a non linearity condition on the flux.
They also conjectured a better Sobolev exponent.  
 In this framework, we can prove that the Sobolev exponent conjectured in \cite{LPT} cannot be exceeded.  
Indeed, we construct  a sequence of smooth solutions which are exactly uniformly bounded in the Sobolev space conjectured in \cite{LPT}. The uniform Sobolev estimate of this  sequence blows up in all more regular Sobolev spaces.

Notice that  we look for the best {\it uniform } Sobolev exponent for a set of solutions.    
 The smoothness of any individual solution is not studied in this paper.  
 This  point is  discussed  later. 
 
 A very important point to note here is the definition of nonlinear flux. 
Indeed, there are various  definitions (\cite{EE,LPT,CJR,COW}).
In \cite{LPT} they give  well known Definition \ref{defLPT} below and a conjecture about 
the maximal smoothing effect in Sobolev spaces related to the parameter ``$\alpha$``  from their definition. The study of periodic solutions leads to another definitions \cite{EE,CJR}.
We obtain  new  Definition \ref{defgnf} for smooth  flux. It 
  generalizes the definition of  \cite{CJR}.
 For smooth flux, our definition  is equivalent to  classical Definition \ref{defLPT}.
  Definition   \ref{defgnf} gives a way to compute the 
parameter  ``$\alpha$''.  This new definition also shows  that smoothing effects 
for scalar conservation laws  depend on the space dimension. 

To be more precise, the smoothing effect   and the related conjecture in the Sobolev framework are recalled in Subsection  \ref{ssse}.
Sobolev spaces are not sufficient to describe all the properties of the solutions. 
Some comments are given in Subsection   \ref{sssmasb}  for  other approaches.
Finally the outline of the paper are given to close the introduction. 

\subsection{The smoothing effect in Sobolev spaces}\label{ssse}
\smallskip

 \alali
 We look for   Sobolev bounds  for entropy solutions $u(.,.) $ of 
\bqr\label{1.1}
\del_t u + \div_\x \F(u) & = & 0,
\eqr
where
$ t \in [0,+\infty[$,
$\x\in \R^d $, $ u: [0,+\infty[_t\times \R^d_\x \to \R$, 
 $\F: \R \to \R^d$ is a smooth flux function, $\F \in C^\infty(\R,\R^d)$,
 and the  initial data are only bounded in $L^\infty(\R^d_\x,\R)$:
\bqr\label{1.2}
u(0,\x) & = & u_0(\x).
\eqr
Let $\a (u)$ be $\F'(u)$. 
Obviously, when $\F$ is linear: $ \a(u)=\a$ where $\a$ is  a constant vector, 
 $u(t,\x)=u_0(\x - t \, \a)$, there is no smoothing effect.
In \cite{LPT}, it  was first proved a regularizing effect  for{\it  nonlinear} multidimensional  flux $\F$. 
The  sharp measurement of the non-linearity plays a key role in our study. 
Let us recall the classical definition for nonlinear flux  from  \cite{LPT}.
\begin{dfni}{\bf [ Nonlinear flux \cite{LPT}]} \label{defLPT}
\\  Let $M$ be  a positive constant, 
   $\F: \R \rightarrow \R^d$ is said to be  
{\rm nonlinear} on $[-M,M]$  if  
 there exist $ \alpha >  0$ and $C=C_\alpha >0 $ 
 such that for all $\delta >0$
%
%
%
\bqr
   \sup_{\substack{\tau^2+ |\xi|^2=1 }} | W_\delta(\tau,\xi) | & \leq  & C \; \delta^\alpha, \label{defal}   
\eqr
where  $(\tau,\xi) \in S^{d} \subset \R^{d+1}, 
             \mbox{ i.e. } \tau^2+ |\xi|^2=1$,
 and 
  $ | W_\delta(\tau,\xi) | $  is the one dimensional measure 
 of the singular set:
\\  $ W_\delta(\tau,\xi)   := 
        \{|v| \leq  M, |\tau + \a(v)\centerdot \xi| \leq \delta \} 
      \quad \subset \quad [-M,M] $    and   $\a=\F'$.
\end{dfni}
 Indeed, $ W_\delta(\tau,\xi)$ is  a neighborhood of the cricital  value $v$   for the symbol of
 the linear operator  $\mathcal{L}[v]$ in the Fourier direction $(\tau,\xi)$ 
  where 
   $\mathcal{L}[v] = \partial_t + \a(v)\centerdot \nabla_x$. 
  The symbol in this direction is: $i \left(  \tau + \a(v)\centerdot \xi \right)$. 
  This operator  is simply related  to  any smooth solution  $u $ of equation \eqref{1.1}
 by the chain rule formula: 
 $$ \del_t u + \div_\x \F(u)  =\partial_t u + \a(u) \centerdot \nabla_x u = \mathcal{L}[u]  u .$$

$\alpha$ is a degeneracy measurement of the operator $\mathcal{L}$ parametrized by $v$.
 $\alpha$ depends only on the flux $\F$ and the compact set $[-M,M]$: $\al = \al[\F,M]$.
 In the sequel we denote by
   \bqr \label{alsup}
   \al_{\sup} & = & \al_{\sup}[\F,M],  \qquad \mbox{  the supremum of  all $\alpha$ 
 satisfying  \eqref{defal}.}
 \eqr
 $\al$,  or more precisely  $\alpha_{\sup}$, is a  key parameter to describe  the sharp smoothing effect
 for entropy solutions of nonlinear scalar conservation laws. 
For smooth flux the parameter $\al$ always  belongs to $[0,1]$,
 for instance: $\alpha_{\sup} =0$ for a linear flux,
 $\alpha=1$ for strictly  convex flux in dimension one.
 For the first time 
$\alpha_{\sup}$ is characterized  below in section \ref{sAFN}.  
Indeed, for smooth nonlinear flux,   
$\dis \dfrac{1}{\alpha_{\sup}}$  is always an integer greater or equal to the space dimension. 
\\

In all the sequel we assume that $M  \geq \|u_0\|_\infty$ and the flux $\F$ is nonlinear 
 on $[-M,M]$: 
  \bqr \label{al+}
                \alpha_{\sup} & > & 0.       
  \eqr  
When nonlinear  condition \eqref{al+} is true,  
the entropy solution operator associated  with the nonlinear 
conservation law \eqref{1.1}, \eqref{1.2}, 
 \bqre  
        \begin{array}{cccc} \mathcal{S}_t:
                          & L^\infty(\R^d_\x,[-M,M]) & \to & L^\infty(\R^d_\x,[-M,M]) \\
                           &  u_0(.)  & \mapsto & u(t,.),
        \end{array}
 \eqre
 has a regularizing effect:    mapping 
$L^\infty(\R^d_\x,[-M,M])$ into $W^{s,1}_{loc}(\R^d_\x,\R)$ 
for  all $t>0$.
\\
In \cite{LPT}, they proved this regularizing effect for all 
$s \dis < \frac{\alpha}{2+\alpha }$. 
\\
 In \cite{TT} the result is improved for all
 $ s  \dis < \frac{\alpha}{1+2\alpha }$
 under a generic assumption on $\a'=\F''$. 
 
Lions,  Perthame and Tadmor conjectured in 1994 
  a better  regularizing effect, (\cite{LPT}, remark 3, p .180, line 14-17). 
 In \cite{LPT},   they   proposed an optimal  bound  $ s_{\sup} $ for  Sobolev exponents of  entropy solutions:  
\bqr  \label{conjLPT}
          s_{\sup} &=&\alpha_{\sup}.
 \eqr
 That is to say  that $u(t,.)$ belongs in all $W^{s,1}_{loc}(\R^d,\R)$ for all $s < s_{\sup}=\alpha_{\sup}$ and for all $t>0$.
\\
 The  shocks formation implies   $ s < 1$ and $s_{\sup} \leq 1$ since $W^{1,1}$ functions
  do not have shock. 
\\

A main result of the paper is  to  give an insight of the conjecture  \eqref{conjLPT} by bounding  the uniform Sobolev smoothing 
effect $s_{\sup}$ for the whole set of  entropy solutions with initial data  bounded by $M$ in $L^\infty$ : 
 \bqr  \label{goal}
          s_{\sup} &\leq&\alpha_{\sup}.
 \eqr

Some results highlight the conjecture \eqref{conjLPT} or the inequality \eqref{goal} in the one dimensional case.
 But for the multidimensional case and for all smooth fluxes, our examples are new. 
\\

\underline{One dimensional case}:

  In one dimension  (d=1) and for uniformly convex flux  it is well known from Lax and Oleinik  that the entropy  solution 
becomes $BV$, (\cite{O,Laxbook}).    Conjecture \eqref{conjLPT} is true in this case
 since  for all $t>0$,  $u(t,.)$ belongs to  $W^{s,1}_{loc}$ for all $s< 1$. 
 In this case we have    conjecture \eqref{conjLPT}  which is simply: $s_{\sup}=1 =\alpha_{\sup} $. 

 For power law flux: $F(u)=|u|^{1+p}$,
 De Lellis and Westdickenberg  built    entropy piecewise smooth solutions   and proved   \eqref{goal}
  (\cite{DW},  Proposition 3.4 p. 1085). 
  For all  one dimensional nonlinear smooth   fluxes,   new continuous  examples   are  also given in \cite{CJ1}.
Both examples are only justified  for a bounded   time interval. 
 
 Recently,   for more general  convex fluxes   the regularity   $s_{\sup}=\alpha_{\sup}$ 
is reached in $W^{s,1}$   (\cite{JaX09}) and also $W^{s,1/s}$  (\cite{BGJ6}).  The proofs need a generalized Oleinik condition.
 \\ 

 For the  class of solutions   with bounded entropy production,  the optimal smoothing effect is proved in  \cite{DW, GP}.
 In \cite{GP} the result  is restricted  for uniform convex flux and in the one dimensional case. 
 This class of solution is larger than the class of entropy solutions. 
For instance,  the uniqueness of solution  for initial value problem  \eqref{1.1}, \eqref{1.2} is not true in general .
Thus the smoothing effect expected is smaller:
 the optimal Sobolev exponent  for uniform convex flux  is only $s=1/3$, (\cite{DW,GP}), instead of $s=1$ for entropy solutions.
\\

\underline{Multidimensional case}:

 For the first time,   the multidimensional case  is investigated to highlight inequality \eqref{goal}. 
  Furthermore, all smooth  nonlinear fluxes are considered in this paper.
 Examples of family of solutions exactly uniformly bounded in $W^{s,1}_{loc}$ 
 with the conjectured maximal exponent  $s=\al_{\sup}$ and  with no improvement of the Sobolev exponent. 

 High frequency periodic solutions  of \eqref{1.1}  are used for this purpose 
 Near a constant state and for $L^\infty$ data,   geometric optics expansions with various frequencies  and various phases  are validated   
in the framework of weak entropy solutions   and  of $L^1_{loc}$ convergence  in \cite{CJR}.    
Here, results of \cite{CJR} are  specified  in $C^1$ for a well chosen phase and proved for  a particular  
 sequence of  smooth solutions 
(without shocks on a strip). 
This allows to  give   a proof of  inequality  \eqref{goal} for  the  ball of $L^\infty$ initial data: $L^\infty(\R^d,[-M,M])$

\subsection{Other approaches for the smoothing effect}\label{sssmasb}
\smallskip
\alali

The  maximal Sobolev exponent is not sufficient 
 to get  all the properties of  entropy solutions. 
 Other relevant ways are indicated.  

In  the 50', Oleinik (\cite{O}) obtained her famous one-sided Lipschitz condition.
 This condition ensures the uniqueness and the $BV$ regularity of the  entropy solution. 
This is the first basis and the proof  of  conjecture \eqref{conjLPT} for one dimensional uniformly convex flux.
Dafermos (\cite{Da85,Da}), with his generalized characteristics,  handled convex and some non convex fluxes. 
Hoff  extended this one-sided condition in several space variables (\cite{Hoff}) but related to a  convex assumption on the flux.
The generalized Oleinik  condition is the key assumption   to prove the best  $W^{s,1}$ smoothing effect 
 in \cite{JaX09}. 
The maximal $W^{s,p}$ smoothing effect is proved in \cite{BGJ6} with 
a one-sided Holder condition and fractional $BV$ spaces, see remark \ref{rind} below.
For a  recent generalization of Oleinik condition for a flux with one inflection point,  we refer the reader to  \cite{JS}.

In the 90', the kinetic formulation of conservation laws (\cite{LPT}) gave another approach. 
It began in 2000  (\cite{Ch00}). 
Some  trace properties were obtained in  \cite{V,DOW,DR,COW}.
These structure of a $BV$  function  for solutions  cannot be given by Sobolev regularity. 
These results are  indeed valid for solutions with bounded entropy production. 
Thus this method necessary  misses some other properties of entropy solutions.   
\\

    The paper is organized as follows.
   In section \ref{sHFW}  
  examples of  highly oscillating solutions are expounded under  new  orthogonality conditions between the flux derivatives and the phase gradient. 
    In section \ref{sAFN},     these orthogonality conditions lead to a new  definition of nonlinear smooth flux.
                 The concept of flux non-linearity is clarified , characterized and compared with other classical definitions. 
Section   \ref{sSE}  is devoted to get optimal Sobolev estimates on 
   oscillating solutions built in section \ref{sHFW}. It is a quite technical part.
  Finally, the section \ref{scgo}  gives the super critical geometric optics expansion with the highest frequency related 
  to the geometric structure of the nonlinear flux. 
        This family of  high frequency waves   highlights   inequality  \eqref{goal}. 
  But the conjecture \eqref{conjLPT} is still an open problem.
 
 
\section{High frequency waves with small amplitude  \label{sHFW}}
\smallskip
 The section \ref{sHFW} deals with highly  oscillating initial data 
  near a constant state:
\begin{equation}\label{1.2os}
\ue(0,\x)=u_0^\eps(\x):=\underline{u}
 +\eps \dis U_0\left( \frac{\vv \cdot \x}{\eps^{\gamma}}\right),
\end{equation}
where $U_0(\theta)$ is a one periodic function w.r.t. $\theta$, $\gamma >0$,  
$\underline{u}$ is a constant ground state, $\underline{u} \in [-M,M]$, 
$\vv \in \R^d $.
The case  $\g=1$ 
is the  classical geometric optics for scalar conservation laws,
  (\cite{DM}). 
In this paper we focus on {\it critical oscillations}
  when $ \gamma > 1$.

The aim of this section is to understand when such high frequency are propagated  or not propagated.
 As we will see, it depends on  new compatibility conditions between the phase and the flux \eqref{Hgamma}.

One of the two following  
asymptotic expansions  (\ref{bsurcri}) or  (\ref{bkill}),
is 
expected   in $ L^1_{loc}(]0,+\infty[ \times \R^d,\R) $ 
 for  the entropy-solution 
 $\ue$ of conservation law  \eqref{1.1} with highly oscillating data 
 \eqref{1.2os}
 when $\eps$ goes to $0$, 
\bqr \label{bsurcri}
      \ue(t,\x) & =&  
\underline{u}+\eps U\left (t,  \frac{\phi(t,\x)}{\eps^\g} \right)
 + o(\eps)
   \\
 \label{bkill}
 \mbox{ or }\qquad
 \ue(t,\x)  &=&  
\underline{u}
 +\eps \overline{U}_0 
 + o(\eps) , 
\eqr
where  the profile $U(t,\theta)$ 
 satisfies a conservation law 
with initial data $U_0(\theta)$,
 $\dis 
   \overline{U}_0= \int_0^1 U_0(\theta)d\theta$ 
and
the phase $\phi$ satisfies the  eikonal equation:
\bqr \label{eqphase}
     \del_t \phi + \a(\underline{u}) \centerdot  \nabla_\x \phi & = 0,
   &
  \phi(0,\x)= \vv \cdot \x.
\eqr 
 Thus the phase is simply a linear phase: 
 \bqr \label{eqphaselin}
   \nonumber 
   \phi(t,\x) & = & \vv \cdot ( \x  -  t \; \a(\underline{u}) ).
\eqr 
 
 The propagation of such oscillating data is obtained under 
 the crucial  compatibility condition (\ref{Hgamma}) below.
 Otherwise, when the  the compatibility condition (\ref{Hgamma}) 
 is nowhere satisfied, the nonlinear semi-group 
 associated with  equation (\ref{1.1}) 
 cancels these too high oscillations, see Theorem \ref{Tptgamma}.
  The  validity or  invalidity   of assumption \eqref{Hgamma} is a key point  related  to the nonlinearity of the  flux (section \ref{sAFN}).
\begin{thm}{\bf [Propagation of  smooth high oscillations]}\label{Tpgamma}
 \alali 
  Let $\gamma$ belong to $]1,+\infty[$
  and let   $q$  be the integer  such that $q-1 < \gamma \leq q$.
\\
 Assume $\F$ belongs to $C^{q +3}(\R,\R^d)$,  
 $U_0 \in C^1(\R/\Z,\R)$, $\vv \neq (0,\cdots,0)$ and  
 \bqr \label{Hgamma}
  \dis  \a^{(k)}(\underline{u})\centerdot \vv = 0,
 &   \qquad k=1,\cdots,q-1 
 \eqr
then there exists $T_0>0$ such that, for all $\eps \in ]0,1]$,
the  solutions of conservation law \eqref{1.1} with initial oscillating data 
 \eqref{1.2os} are smooth on $[0,T_0]\times \R$ and
\bqre
 \ue(t,\x) = 
\underline{u}+\eps U\left (t,  \frac{\phi(t,\x)}{\eps^\gamma} \right)
 + \mathcal{O}(\eps^{1+r}) \mbox{ in }  C^1([0,T_0]\times \R^d),  
\eqre
where 
$ \dis  0 < r= \left \{ \begin{array}{ccc}
     1 & if & \gamma = q, \\
     q-\gamma & else, &   
   \end{array}
   \right.
$
 and the smooth profile $U$ is uniquely determined by the Cauchy problem    \eqref{eqUpgamma},  
  $\phi$ is given by the eikonal equation \eqref{eqphase}:  
\begin{eqnarray} 
\label{eqUpgamma}
 \frac{\partial U }{\partial t} + b \frac{\partial  U^{q+1}  }{\partial \theta} = 0,  
 & \quad &
  U(0,\theta)  = U_0(\theta), 
\end{eqnarray}
with  $    \dis  b= \left \{ \begin{array}{ccc}
    \frac{1}{(q+1)!}  \left( {\a^{(q)} (\underline{u})\centerdot \vv}   \right) 
    & if & \gamma = q, \\
    0 & else. &   
   \end{array} \right.$.
\end{thm}
We deal with smooth solutions to  compute later  Sobolev bounds.
Indeed,  the  asymptotic stays valid after shocks formation   and 
 for all positive time but in $L^1_{loc}$  instead of $L^\infty$ (\cite{CJR}).


 When $\g = 1$, we do not need  assumption  \eqref{Hgamma}.
 It is  the classic case for geometric optics (\cite{DM,R}). 

In dimension  $d\geq 2$,  
 it is always possible to find a non trivial vector $\vv$
 satisfying  \eqref{Hgamma}. 
 At least for $\g=2$, \eqref{Hgamma} is  reduced  to  find $\vv \neq 0$  such that 
 $\a'(\underline{u}) \centerdot \vv = 0$.
Thus, such singular solutions  always  exist in dimension greater than one. 
 But, for genuine nonlinear  one dimensional conservation law, there 
 is never such solution.  
Of course, we assume $U_0$  be a non constant function  and
 $\F$  be   a  nonlinear function  near $\underline{u}$,
 else the theorem is obvious.  Indeed, when  $U_0$ is constant,   $\ue$  is also constant. 
 When  $\F$ is linear on $[\underline{u} - \delta, \underline{u} + \delta]$
 for some $\delta > 0$, high oscillations propagate for all time  without any restriction  of the phase and  of  the frequency size.

In fact,  Theorem \ref{Tpgamma} expresses
  a kind of degeneracy 
of {\bf multi}dimensional scalar conservation laws.
 This degeneracy (period smaller than the amplitude)  appears for quasilinear systems
  whit some nonlinear degenerescence 
(see for instance \cite{CG}). 

 Notice that for $\g > 1$, smooth solutions exist for  larger time   than  
 it is currently  known \cite{Da,Laxbook}:
 $ T_\eps \sim 1/|\nabla_\x u_0^\eps| \sim \eps^{\g-1}$.
  Furthermore, equation \eqref{eqUpgamma} is nonlinear if and  only 
if $\g \in \N$ and $\a^{q}(\underline{u}) \centerdot \vv \neq 0$.
\\

\noindent
\bpro
First  one performs  
a WKB computations  with following ansatz: 
 \bqr
   \label{eqansatz} 
   \dis \ue(t,\x) = \underline{u} &+&
   \eps  \; \Ue\left (t,  \frac{\phi(t,\x)}{\eps^\gamma} \right).
\eqr 
Notice that we use for the proof the exact profile $\Ue$ as in \cite{Ju2}. 
It is a method to  sharply control the difference between  the exact solution 
and the geometric optics expansion: $\Ue$ and $U$. 
\\
The Taylor expansion of the flux and the remainder are:
\bqr
 \nonumber 
 \F(\ue) & =& \sum_{k=0}^{q+1} \eps^k \frac{\Ue^k}{k!} 
            \F^{(k)}(\underline{u}) +\eps^{q+2}\G^\eps_q(\Ue),
\\  \nonumber 
 \G^\eps_q(U) & =& U^{q+2}\int^1_0 \frac{(1-s)^{q+1}}{(q+1)!}
 \F^{(q+2)}(\underline{u}+s \eps U) ds, 
\\  \nonumber 
 g^\eps_q(U) & = & \vv. G^\eps_q(U).
\eqr
We now compute the partial derivatives with respect to time and space variables:
\bqr
 \nonumber 
  \del_t \Ue\left (t,  \frac{\phi(t,\x)}{\eps^\gamma} \right)  & =&   \del_t \Ue 
              - \eps^{-\gamma} (\a(\underline{u})\cdot \vv ) \del_\theta \Ue 
\eqr
\bqr
 \nonumber 
\div_\x \F(\ue)  & = &     
   \sum_{k=0}^{q} \eps^{k+1-\gamma} \frac{\del_\theta \Ue^{k+1}}{(k+1)!} 
            \a^{(k)}(\underline{u})\cdot \vv
        + \eps^{q+2} \div_\x \G^\eps_q(\Ue)
 \\  \nonumber 
       & = &
     \eps^{1-\gamma} (\a(\underline{u})\cdot \vv ) \del_\theta \Ue
   + \eps^{q+1-\gamma} c_q \del_\theta \Ue^{q+1}
   + \eps^{q+2-\gamma} \del_\theta g^\eps_q( \Ue),
   \eqr
where $ \dis c_q = \frac{ \a^{(q)}(\underline{u})\cdot \vv}{(q+1)!} 
            $.
Then  simplification yields
\bqr \label{eqve}
   \del_t \ue + \div_\x \F(\ue) 
 &
 =&   
\eps \left ( \del_t \Ue  + \eps^{q-\gamma} c_q \del_\theta \Ue^{q+1} 
        + \eps^{1+q-\gamma} \del_\theta g^\eps_q(\Ue) \right). 
 \eqr
It suffices to take $\Ue$ solution of the one dimensional scalar conservation laws
 with $\psi_\eps(U)= \eps^{q-\gamma} c_qU^{q+1} 
    + \eps^{1+q-\gamma}  g^\eps_q(U) $
\bqr 
 \label{eqUe}
\del_t \Ue  +   \del_\theta   \psi_\eps(\Ue) = 0, 
  & \qquad &
 \Ue(0,\theta)=U_0(\theta). 
\eqr
Notice that $\psi_\eps = O(1) \in C^2_{loc}$. 
For $\gamma <  q$, $\psi_\eps$ is even smaller:
$\psi_\eps = O(\eps^r) \in C^2_{loc}$.  
That is enough to prove the existence of a sequence of  smooth oscillating solutions on the same strip.
\\
\underline{Uniform life span for smooth solutions $(\Ue)_{0 < \eps \leq 1}$}:
\\
We use the method of characteristics  with $  \psi_\eps'(U) =  \frac{d}{d U}\psi_\eps(U) $: 
\bqre 
    \frac{d}{dt} \Theta(t,\theta) = \psi_\eps' (U_\eps ( t, \Theta(t,\theta)), & \quad & \Theta(0,\theta)=\theta.
\eqre
Since $U_\eps$ is constant along the characteristics, $\Theta(t,\theta)= \theta + t \psi_\eps' (U_0(\theta))$.
As soon as  the map $ \theta \rightarrow\Theta(t,\theta)$ is not decreasing no shock occurs. 
\bqre 
   \frac{\partial}{\partial \theta} \Theta (t,\theta)  & = &  1  +  t \psi_\eps''(U_0(\theta))  \frac{d}{d \theta} U_0 (\theta)
\eqre
The first shock appears  at the time $T_\eps$ when the right hand side vanishes. 
Let  $m_0 = \dis  \sup_{[0,1]} |U_0| > 0$,   
       $d_0 = \dis \sup_{[0,1]}   \left |   \frac{d}{d \theta} U_0   \right|  $,
        $m = \dis \sup_{0 < \eps\leq 1}  \sup_{  | U - \underline{u}| \leq m_0} | \psi_\eps''(U)|   $, 
\bqre 
       1/T_\eps & =&   \sup_{[0,1]} \left( -  \psi_\eps''(U_0(\theta))  \frac{d}{d \theta} U_0 (\theta) \right )
                      \leq   m \; d_0.  
\eqre
Of course, for constant initial data ($d_0=0$),  no shock occurs,  the solution is constant and $T_\eps = + \infty$. 
In general $m\;  d_0 \neq 0$, $T_\eps$ is finite   but  $ 0 < \inf_{0<\eps \leq 1}T_\eps $ 
since  $T_\eps  \geq   1/(m \; d_0)$ for all $0 < \eps \leq 1$. 

 This gives the existence of a positive time   $T_0 < T^*=\inf\{T_\eps, \; \eps \in]0,1]\} $ 
 such  that $\Ue \in  C^1([0,T_0]\times \R/\Z)$.
Thus  
 $\ue$, which is  well defined by \eqref{eqansatz}, 
belongs to  $ C^1([0,T_0]\times \R^d)$ for all $0 <\eps \leq 1$.
\\

Now we prove the $C^1$ convergence of the geometric optics expansion.
There are two cases: $\gamma $ is an integer or not.
\\
\underline{$q=\gamma$}: From \eqref{eqve} and \eqref{1.1} we get 
  \bqre
 \del_t \Ue  +   \del_\theta \left ( c_q \Ue^{q+1} 
        + \eps^{}  g^\eps_q(\Ue) \right)= 0 ,  
   & \qquad &
    \del_t U  +  c_q \del_\theta  U^{q+1} = 0, 
 \\
 \Ue(0,\theta)=U_0(\theta), &   & 
 U(0,\theta)=U_0(\theta). 
 \eqre     
  The  method of characteristics  gives $C^1$ characteristics, 
 $C^1$ solutions and 
\bqre 
  \|\Ue - U\|_{C^1([0,T_0]\times \R^d)} & = & O(\eps),
 \eqre 
 where $$\| U\|_{C^1([0,T_0]\times \R^d)}= \|U\|_{L^\infty([0,T_0]\times \R^d)}
 +\|\del_t U\|_{L^\infty([0,T_0]\times \R^d)} +\|\del_\t U\|_{L^\infty([0,T_0]\times \R^d)}.$$
\noindent
 \underline{ integer {$ q > \gamma $}}:  The proof is similar 
 except  the term $ \eps^r c_q \del_\theta \left ( c_qU^{q+1}  \right) $ becomes a remainder, 
with $r=q-\gamma$
 and $U(t,\t)=U_0(\t)$, 
 thus 
\bqre 
  \|\Ue(.,.) - U_0(.)\|_{C^1([0,T_0]\times \R^d)} & = & O(\eps^r),
 \eqre 
which concludes the proof. 
\epro

 When condition (\ref{Hgamma}) is violated,   
oscillations are  immediately canceled.
\begin{thm}{\bf [Cancellation of  high oscillations, \cite{CJR}]}\label{Tptgamma} \alali
 Let $\F$ belong to $ C^{q+2}$ and $U_0 \in L^\infty(\R/\Z,\R)$,
 where $q-1< \gamma \leq q$  where  $q$  is defined in Theorem \ref{Tpgamma}.
 If for some $  0 < j < q$ 
 \bqr \label{Htgamma}
\dis   \a^{(j)}(\underline{u})\centerdot \vv &\neq& 0  
 \eqr
then  the solutions $\ue$ of conservation law \eqref{1.1} 
with initial oscillating data 
 \eqref{1.2os} for $\eps \in ]0,1]$ satisfy when $\eps \rightarrow 0$
 \bqre
 \ue(t,\x)
 &  = &  
\underline{u}+\eps \overline{U}_0
 + \mbox{o}(\eps)   \qquad  \mbox{ in } L^1_{loc}(]0,+\infty[\times \R^d).  
\eqre
\end{thm}
Obviously the interesting case is when $U_0$ is non constant.
In this context, when $U_0$ is smooth and non constant 
the first time when a shock occurs $T_\eps \to 0$ when $\eps \to 0$.
Thus solutions are weak entropy solutions.
\\
The proof is in the spirit of  \cite{CJR} and uses averaging lemmas (see  \cite{P} and the  references given there). 
The proof  is briefly expounded to be  self-contained. 
\\

\noindent 
\bpro
 For non constant initial data  it is impossible to  avoid shock waves on any fixed strip $[0,T_0]\times\R^d$
 with $T_0 >0$
 as in the previous proof of Theorem \ref{Tpgamma} since the time span 
 of smooth solutions is $ \eps^\beta$ where $\beta=\gamma - j > 0$.
 
First, with a change of space variable
 $\x \leftrightarrow \x - t. \a(\underline{u})$, 
   we can assume that $\a(\underline{u})=0$.

The WKB  computations  use the following anzatz: 
$\ue(t,\x)= \underline{u} + \eps \ve(t,\x) $ where
$\ve(t,\x)= \We (t, \eps^{-j} \phi(t,\x) )$.  
Indeed, the condition \eqref{Htgamma} leads to such anzatz as we can see in the WKB computations of the proof of Theorem \ref{Tpgamma}.
Then $\We$ satisfies the one dimensional  nonlinear conservation laws:  
 \bqr  
  \label{eqWeps}
 \del_t \We  +   \del_\theta \left ( c_j\We^{j+1} 
        + \eps^{}  g^\eps_j(\We) \right)= 0 ,   
    & \We(0,\theta)=U_0(\eps^{-\beta}\theta),
   &  c_j \neq 0 . 
\eqr
 $\We(0,.)$ converges weakly towards $\overline{U}_0$.
 As in \cite{CJR}, $\We$ is relatively  compact in $L^1_{loc}$ 
thanks to  averaging lemmas.
Then 
$\We$ converges towards  the unique entropy solution $W$ of
\bqre 
    \del_t W  +  c_j \del_\theta  W^{j+1} = 0, 
 & \qquad  & 
 W(0,\theta)=\overline{U}_0. 
 \eqre  
That is to say that $  W(t,\t) \equiv \overline{U}_0$.
Then 
$\ve(t,\x)$  
 converges towards $\overline{U}_0$ in  $L^1_{loc}$ which 
 concludes the proof.
\epro 



\section{Characterization of  nonlinear flux \label{sAFN} }
\smallskip

The flux nonlinearity is characterized by   the  parameter $\alpha$ in Definition \ref{defLPT}, the Lions-Perthame-Tadmor definition of nonlinear flux. 
The smoothing effect  depends only on  the best $\alpha=\alpha_{\sup}$. 
 The understanding  of the  parameter $\alpha_{\sup}$  is   a key step to the comprehension of the regularity of entropy solutions. 
 Unfortunately,  there are only few examples where $\alpha_{\sup}$ is computed in dimension $2$ 
 (\cite{LPT, TT}) and there are some remarks in \cite{Ja09,JaX09,BJ}.

For the first time,  for all smooth fluxes  and for all dimensions   we characterize the fundamental  parameter  $\alpha_{\sup}$.
For this purpose we state    Definition \ref{defgnf}  of smooth nonlinear flux.  
This new definition is related to  the critical geometric optics  expansion given  in Section \ref{sHFW}.
Let us emphasize on  three important consequences of Definition \ref{defgnf}.
 \begin{itemize}
\item  The parameter $\alpha_{\sup}$ is explicitly characterized with the  flux derivatives in Theorem \ref{thmal}.
\item    The super critical geometric optics  expansion is built in  Theorem \ref{thsgo}.
\item    
The uniform  maximal smoothing effect is highlighted in section \ref{sconj}.
\end{itemize}

    We explain this new definition in the subsection  \ref{ssnsf}. 
   We compare our new definition with some other classical definitions in subsection \ref{sscomparison}.
    We   prove that all definitions of nonlinear flux  are equivalent for analytical flux.

\subsection{Nonlinear smooth flux} \label{ssnsf}
\alali

%

We introduce a new definition of nonlinear $C^\infty$ flux  related to critical geometric optics expansions.
When  the compatibility conditions \eqref{Hgamma} are satisfied  in Theorem \ref{Tpgamma}, very high frequency waves are smooth solutions of the conservation law \eqref{1.1}.
Furthermore,  these conditions are optimal thanks to Theorem \ref{Tptgamma}.
What is the highest frequency waves  as in   Theorem \ref{Tpgamma} solutions of \eqref{1.1}? 
Indeed,  near the constant state $\underline{u}$ 
 we can propagate waves with frequency $\eps^{-m}$, $m>1$,
if  the set $\{\a'(\underline{u}),\a''(\underline{u}),\cdots,\a^{(m-1)}(\underline{u}) \}^\perp$ is not reduced to $\{0\}$. 
Thus the maximal $m$ occurs when
$  \{0\} =\{\a'(\underline{u}),\a''(\underline{u}),\cdots,\a^{(m)}(\underline{u}) \}^\perp $ 
and $ \{0\} \neq \{\a'(\underline{u}),\a''(\underline{u}),\cdots,\a^{(m-1)}(\underline{u}) \}^\perp $ .
We now can write the following definition.
 
\begin{dfni}{\bf[Nonlinear smooth  flux]}  \label{defgnf}\alali
 Let  the flux $\F $ belong to $  C^\infty(\R,\R^d)$ and $I=[-M,M]$.
The flux is said to be  {\bf nonlinear}  on $I$ if,
   for all $u \in I$,   
there exists $m \in\N^* $ such that 
 \bqr \label{defnl} 
rank\{\a'(u),\cdots,\a^{(m)}(u)\}& = & d.
 \eqr
Furthermore, the flux is said to be {\bf genuine nonlinear} if
$m=d$  is enough in \eqref{defnl} for  all $u \in I$.  
\end{dfni}
\noindent
 As usual, the non-linearity is a matter of the second derivatives of $\F$,  $ \a'= \F'' $.
 Notice that $m \geq d $. We need at least $d$ vectors  in \eqref{defnl} to span the space $\R^d$.
 Thus the genuine nonlinear case is the strongest  nonlinear case.

 The genuine nonlinear case was first stated in \cite{CJR} (condition (2.8) and Lemma 2.5 
 p. 447 therein).  The genuine nonlinear condition  in  the $d$ dimensional case
 \bqr
 \det(\a'(u), \a''(u), \cdots, \a^{(d)}(u))  & \neq &   0,   \quad \forall u \in I,   \label{gnf}
 \eqr
 was also in  \cite{COW}, see condition (16) p. 84 therein.
 The simplest  example  of genuine nonlinear flux $\F$  with   the velocity $\a$ 
was given in \cite{CJR,COW,BJ}:
  \bqre \a(u)& =& (u,u^2,\cdots,u^d)   
    \qquad \mbox{ with }\F(u) = \left(\frac{u^2}{2},\cdots,\frac{u^{d+1}}{d+1}\right) .
  \eqre
 Definition \ref{defgnf} is a generalization of the genuine nonlinear condition \eqref{gnf}.
 Definition \ref{defgnf} is more explicit  with  following integers with 
 $I=[-M,M]$. 
\bqr  \label{defmu} 
      d_\F[u] & = & 
      \inf\{ k \geq 1, rank\{\F''(u),\cdots,\F^{(k+1)}(u)\}= d \} \qquad  \geq d, \\
         \label{defm} 
       d_\F  & = & \dis \sup_{|u|\leq M}  d_\F[u]  \qquad \in \{d,d+1, \cdots\}  \cup \{+\infty\}.
\eqr
Indeed, Definition \ref{defgnf}  states  that  the flux is genuine nonlinear 
when  $ d_\F$ 
reaches its minimal value, $ d_\F=d$.
\\
Conversely, when  the flux $\F$ is linear,   $\a$ is  a constant vector   in $\R^d$
 and  $ d_\F$ reaches its maximal value, 
 $ d_\F=+\infty$. 
\\
 Between $ d_\F=d$ and $ d_\F=+\infty$, there is a large  variety of nonlinear flux.
\\

    The following theorem gives  the optimal parameter  $\alpha$  \eqref{defal}  for smooth flux.
\begin{thm}{\bf[Sharp measurement of the flux  non-linearity ]}
   \label{thmal}
 \alali
 Let $\F$ be a smooth flux, $\F \in C^\infty([-M,M],\R^d)$, the measurement
 of the flux  non-linearity  $\alpha_{\sup}$ is given by
 $$  
    \begin{array}{ccccc}
        \al_{\sup} & = & \dis  \frac{1}{ d_\F} & \leq & \dis \frac{1}{d}.
   \end{array}
 $$
 Furthermore, when  $\al_{\sup} > 0$ there exists $\underline{u} \in [-M,M]$
  such that $ d_\F=d_\F[\underline{u}]$.
\end{thm}
  A similar result for the genuine nonlinear case: $ d_\F=d$, can be found in \cite{BJ}.
 
This theorem is a powerful tool to compute the parameter $\al_{\sup}$, for instance:
\begin{itemize}
\item  $F(u)= (\cos(u), \sin(u))$  is genuine nonlinear flux , $\al_{\sup}=1/2$ 
 since $\det(F''(u),F'''(u))=1$.
 \item When  $F$ is polynomial  with degree less or equal to the space dimension 
$d$,  $\al_{\sup}=0$ and $F$ does not satisfy Definition \ref{defgnf}. 
\item 
It is well known that the ``Burgers multi-D'' flux 
 $F(u)=(u^2,\cdots,u^2)$
 is not nonlinear  when $d\geq 2$.
Let us explain this fact by  two arguments: the explicit computation of $\alpha_{\sup}$ 
and  a sequence of high frequencies waves solutions of \eqref{1.1}.
 \begin{itemize}
  \item 
  $\a''(u) \equiv 0$ so  $d_\F = + \infty$ and  Theorem \ref{thmal} yields $\alpha_{\sup}=0$.
 \item 
The sequence of oscillations with large amplitude $(u_\eps)_{0<\eps\leq 1}$
given by $u_\eps(t,\x)= u_0^\eps(\x)=\sin\left(\frac{x_1 - x_2}{\eps} \right)$ 
  blows up in any $W^{s,1}_{loc}$ , $s>0$:  for all $t$, 
$ 
\displaystyle{  \sup_{0 < \eps \leq 1} \|u_\eps(t,.)\|_{W^{s,1}([0,1]^d,\R)} = + \infty.}
$
But the sequence of initial data is uniformly bounded in $L^\infty$, $\| u_0^\eps\|_{L^\infty} = 1$. 
Thus there is no improvement of the uniform initial Sobolev bounds.
\end{itemize}
%
%
 \item When  $F$ is polynomial such that $deg(F_i) = 1+i $,
    $F$ is genuine nonlinear:  $\dis \al_{\sup} =  \frac{1}{ d} $ .  

\end{itemize}

   \begin{remark}      For smooth Flux  $\alpha_{\sup}$ is the inverse of an integer. 
          Not all  real value of $\alpha_{\sup}$  in $[0,1]$ are possible for $\F\in C^\infty$. 
        With less smooth flux, all other values of $\alpha_{\sup}$ are possible
       (\cite{LPT,DW,TT,JaX09,CJ1, BGJ6}).
  \end{remark}


We  now begin the proof of  Theorem \ref{thmal}  related  to  some proofs of phase stationnary lemmas (\cite{St,JaX09,BJ}). 
The proof  needs many lemmas. 
 First we recall  Lemma 1 p. 125  from \cite{BJ} giving the optimal $\al$
 for real  functions. 
\begin{lem}[\cite{BJ}] 
 \label{mphi}
 Let $\ph \in C^\infty([-M,M],\R)$, 
 \bqre
       m_\ph[v] &= & \inf \{ k \in \N,\, 
       \ph^{(k)}(v) \neq 0 \}   
 \qquad   
 \in \overline{\N}= \N \cup \{+\infty\}, \\
  m_\ph  & = & \sup_{|v|\leq M } m_\ph[v]  \quad  \in \overline{\N}, \\
    Z(\ph,\eps)  & = & \{ v \in [-M,M],\; |\ph(v)|\leq \eps\}.
 \eqre
If $0 <m_\ph<+\infty$ 
then there exists $C> 1 $  dependent of the function  $\phi$ such that, for all $\eps \in ]0,1]$,
  \bqr  \label{mphineq}
     C^{-1} \eps^\alpha \leq  meas(Z(\ph,\eps))\leq  C \eps^\alpha 
  & \mbox{  with  } & 
 \dis \alpha =\frac{1}{m_\ph}.
 \eqr
\end{lem}

  To compute the measure of $Z(\ph,\eps) $ with a  different assumption, 
we adapt     a  proof of E. Stein about  stationary phase method \cite{St}.  
The main point in the  following lemma is that the constant does not depend  on the function $\phi$. 
Indeed, the condition   $1 \leq  |\phi^{(k)}(v)|$ is stronger than  the condition $m_\ph= k$.
The following  lemma is fundamental to prove Theorem \ref{thmal}.
\begin{lem} \cite{BJ} \label{Ls2}
Let $k \geq 1$, $I$ an interval of $\R$,
    $ \phi \in C^k(I,\R)$.
 \bqre 
 \begin{array}{lrclc}
  \mbox{ If  } &
 1 &\leq & |\phi^{(k)}(v)| ,  & \mbox{ for all } v \in I,   \\
\mbox{ then } 
 & \mbox{measure}\{v \in I,\;  |\phi(v)|  \leq  \eps\}  &\leq & \overline{c}_k  \; \eps^{1/k}, &
 \end{array}
\eqre
where  $\overline{c}_k$ are constant independent of $\phi$. 
\end{lem}
\bpro
Since the result is independent of the  interval I and the constant sign of the derivative $\phi^{(k)}$ on the interval, let us suppose that $ I = \R$  and
$ \phi^{(k)} (v) \geq 1$   for all  $ v \in \R$. 
 Thus   we have   for all $ v \geq u $ :  $ \phi^{(k-1)} (v)  -  \phi^{(k-1)} (u) \geq v - u $ . 
 This inequality   shows that the function $\phi^{(k-1)}$ admits an unique root. 
Assume  $\phi^{(k-1)}(0) = 0$ without loss of generality.

With these assumptions we  prove the lemma  when $ k = 1$. 
 Since  $|\phi(v)|\geq    |v|$ for all $v$, we have 
  $Z(\phi,I, \eps) = \{v \in I,\, |\phi(v)| \leq \eps \}\subset [-\eps,\eps ]$.
So the lemma is proved for $k=1$ with $\overline{c}_1=2$.

We now prove the Lemma by induction on $k>1$. 
 We have  for all $v$,  $|\phi^{(k-1)}(v)| \geq  |v|$.
Let $\eta> 0 $ . 
 Notice that $\mbox{meas}(Z(\phi,[-\eta,\eta], \eps)) \leq 2 \eta$.
Let $\psi$ be the function $\phi/\eta$.
Notice that $\psi^{(k-1)}(v) \geq 1$ on $]\eta, +\infty[$.
  By our inductive hypothesis on $\psi$ we have
$\mbox{meas}(Z(\psi,]\eta, + \infty[, \eps) \leq \overline{c}_{k-1} (\eps)^{1/(k-1)}$, so 
$\mbox{meas}(Z(\phi,]\eta, + \infty[, \eps) \leq \overline{c}_{k-1} (\eps/\eta)^{1/(k-1)}$. 
\\
A similar argument yields $\mbox{meas}(Z(\phi,]- \infty, -\eta[, \eps) \leq \overline{c}_{k-1} (\eps/\eta)^{1/(k-1)}$. 
These previous three bounds gives  
$ \mbox{meas}( Z(\phi,\R, \eps)) \leq 
  \dis  g(\eta)=  2 \left(\eta + \overline{c}_{k-1} (\eps/ \eta)^{1/(k-1)}\right)$.
This last inequality is valid for all $\eta >0$.  It suffices to minimize the function $g$ on $]0,+\infty[$.
A  computation of the minimum   yields
 $
   \mbox{meas}( Z(\phi,\R,\eps)) \leq 
  \dis  \overline{c}_{k}\eps^{1/k}$,
 where $  \overline{c}_{k}= 4 \left( \overline{c}_{k-1} /(k-1)\right)^{(k-1)/k}   $  
which concludes the proof.
 \epro 

The previous lemma  is generalized with parameters in a compact set, see Lemma 4 p. 127 in \cite{BJ}.
\begin{lem}[\cite{BJ}]\label{Lscompact}
 Let $P$ be a compact set of parameters, 
$k$ a positive integer, 
$A > 0$, $V=[-A,A]$,  $K= V \times  P$,
 $\phi(v;p) \in C^0(P,C^{k}(V, \R))$, 
 such that, for all $(v,p)$ in the compact $K$, we have 
\bqre 
 \label{cond2}
\dis 
  \sum_{j=1}^{k} 
 \left| \frac{\partial^j \phi}{\partial v^j} \right| (v;p)& >& 0.
 \eqre
 Let $ Z(\phi(.;p),\eps)=\{v \in V,\;| \phi(v;p)| \leq \eps\}$, 
then  there exists  a constant $C$ such that
\bqre   
 \displaystyle{  \sup _{p\in P}} \mbox{meas}(Z(\phi(.;p),\eps)) &\leq &C \eps^{1/k}.
\eqre
\end{lem}

We now turn to  the key integer $d_\F$.
\begin{lem}[] \label{ldf}
  If $\F$ is a nonlinear flux on $I$ in the sense of Definition \ref{defgnf} then 
 $d_\F$ is finite and there exists $\underline{u} \in I$ such that 
 $d_\F = d_\F[\underline{u}]$.
\end{lem}
{\bf Proof}
 Let $u$ be fixed in $I$. Then there exits, $1 \leq j_1 < j_2 < \cdots < j_d= d_\F[u]$ such that 
$ rank\{\a^{(j_1)}(u),\cdots,\a^{(j_d)}(u)\}  =  d$
 by  the definition of $d_\F[u]$.   So  the continuous function 
 $g(v) = \det(\a^{(j_1)}(v),\cdots,\a^{(j_d)}(v))$ does not vanish at $v=u$. 
 By continuity, this is still true on an open set $J$ with $u \in J$. 
  Since $j_d=d_\F[u]$, we have $d_F[v] \leq d_F[u]$ for all $v \in J$.  
  Thus $v \mapsto d_F[v]$ is upper semi-continuous and the result follows immediately on the 
 compact set $I$.  
\epro

Now we are able to prove Theorem \ref{thmal}.
\\

 {\bf Proof of Theorem \ref{thmal}}.
There are two steps.

\underline{step 1}: $ \dis  \al_{\sup} \geq \frac{1}{ d_\F} $.

  Set $\phi(v;\tau,\xi)= \tau + \a(v) \cdot \xi$ with $\tau^2 + |\xi|^2=1$.
  $\tau$ and $\xi$ are fixed.  Since $\phi(.;\tau,0)=\tau$ has no roots, 
 we can assume that $\xi \neq 0_{\R^d}$.  
  For $j \geq 1$ we have 
   $\partial^{j}_v\phi(v;\tau,\xi)= \a^{(j)}(v) \cdot \xi $.  By definition of 
 $d_\F[v]$ there exists $ j \leq d_\F[v] \leq d_\F$ such that  
$\partial^{j}_v\phi(v;\tau,\xi) \neq 0$.  
 Thus,  we have when $\xi \neq 0$
 \bqr \label{ineqxineqz}
  \dis   \sum_{j=1}^{d_\F} |\partial^{j}_v  \phi(v;\tau,\xi)| > 0.
 \eqr   
 When $\xi=0$, we have  $\tau= \pm 1$ since $\tau^2 + |\xi|^2=1.$
 The function $\phi(v;\pm 1,0)= \pm 1 \neq 0$.  By continuity of this function 
 there exists an open neighborhood $V$ of $(1,0_{\R^d})$ such the function does not vanish on $\overline{V}$.   Set $P$ be the complementary set of $V$ in the 
 unit sphere of $\R^{d+1}$.
  $P$ is compact and \eqref{ineqxineqz} is true on $P$. 
 Now we can use Lemma \ref{Lscompact}  to conclude the first step.

\underline{step 2}: $ \dis  \al_{\sup} \leq \frac{1}{ d_\F} $. 

 Take $\underline{u}$ from Lemma \ref{ldf}.  Then  there exists $\xi \neq 0$
 such that  $\partial^{j}_v\phi(v;\tau,\xi)=0$  for $1 \leq j < d_\F$ 
and $\partial^{j}_v\phi(v;\tau,\xi)\neq 0$ for $j=d_\F$.   For such $\xi \neq 0$,
 we choose  $\tau$ such that $\ph(v)=\phi(v;\tau,\xi)$ vanishes at $v=\underline{u}$. Now, by Lemma \ref{mphi}, the second step is proved.
\\

Finally  $ \dis \frac{1}{ d_\F} \leq  \al_{\sup} \leq \frac{1}{ d_\F} $
 and the proof  is complete with Lemma \ref{ldf}.
 \epro

\subsection{Comparisons with other nonlinear flux definitions}\label{sscomparison}
\smallskip
\alali
There are more general  definitions of nonlinear flux \cite{EE,LPT}. 
But the precise smoothing is  related to  Definition \ref{defLPT} or  Definition \ref{defgnf}
 and the parameter $\al_{\sup}$ or equivalently $d_\F$.
Let us compare theses  definitions with  Definition \ref{defgnf}. It can be useful for other applications.
\medskip\\
In \cite{LPT},  there is a more general definition of   nonlinear flux.
\begin{dfni}{\bf [General Nonlinear Flux \cite{LPT}]} \label{defLPTg}
 A flux $\F$, differentiable on $[-M,M]$ is said to be nonlinear if 
 the degeneracy set 
$$ W(\tau,\xi) =\{|v| \leq M, \, \tau + \F'(v)\cdot \xi =0\} $$ 
 has null Lebesgue measure
for all $(\tau,\xi)$ on the sphere.
\end{dfni}
This definition is of a great importance since this condition implies 
the compactness of the semi-group $\mathcal{S}_t$ associated with  the conservation 
 law $\eqref{1.1}$.
\begin{proposition}[]\label{31LPTg}
Let $\F $  be a smooth  flux in $C^\infty$. 
  Assume $\F$ satisfy  Definition \ref{defgnf} then  
  $\F$ is nonlinear for Definition \ref{defLPTg} but 
 the converse can be  wrong.
\end{proposition}

\bpro 
  Lemma \ref{ldf} and 
  Theorem \ref{thsgo} show that nonlinearity of  Definition \ref{defgnf} 
 implies   nonlinearity of  Definition \ref{defLPT} and then of 
Definition \ref{defLPT}.  
 But we can give a direct proof  from 
 Lemma 2.5 and remark (2.3) p. 447 in \cite{CJR}, (see also \cite{COW} p. 84).

 Notice that $ W(\tau,0)=\emptyset$  since $\tau=\pm 1$.
 So we assume  that $\xi\ne 0$. Set $\phi(v)=\tau + \F'(v)\cdot \xi$.
 Since $\phi^{(k)}(v)=\F^{(k+1)}(v)\cdot \xi$, for any $v$,
 there exists $k>0$ such that 
 $\phi^{(k)}(v)\neq 0$ by Definition \ref{defgnf}. 
 So the roots of $\phi$ are isolated and the set $ W(\tau,\xi)$ is finite.
 
Conversely  the counter-example
$\F'(u) = \exp(-1/u^2)(u,u^2,\cdots,u^{d})$   does not 
 satisfies Definition \ref{defgnf}
since $d_\F[0]=+\infty$. 
\\
But $\F$ satisfies Definition \ref{defLPTg}.
Indeed, with $h(v)= \tau \exp(1/v^2) + \xi\cdot(v,v^2,\cdots,v^{d}) $,
 the  set $ W(\tau,\xi) - \{0\}$ is the set of roots of $h(.)$.    
 If $\tau=0$,
 we deal with the genuine nonlinear flux from Definition \ref{def31} and 
 the degeneracy set $ W(\tau,\xi)$ is a null set. Indeed,  it is finite. 
 If $\tau \neq 0$, $h(.)$ is analytic and non trivial  on $\R^*$. 
 Consequently   $ W(\tau,\xi)$ is countable and also a null set 
 which concludes the proof.
\epro

Engquist and E in \cite{EE} gave another definition of strictly nonlinear flux 
 generalizing  Tartar \cite{T}.

\begin{dfni}{\bf [ Strictly Nonlinear Flux \cite{EE}]} \label{defEE}
\\  Let $M$ be  a positive constant, 
   and  $\F : [-M,M] \rightarrow \R^d$ be 
   a function  twice differentiable on $[-M,M]$. 
\\
$F$ is said to be  
{\rm strictly nonlinear} on $[-M,M]$  if    for any 
 sub-interval $I$ of $[-M,M]$, the functions $F_1'',\cdots,F_d^{''}$ are linearly independent on $I$,
\\ i.e.,
  for any  constant vector $\xi$, if  $\xi \cdot \F''(u) =0$ for all $u \in I$
   then $\xi = 0$.
\end{dfni}

\begin{proposition}[]\label{31EE}
Let $\F$  be a   $C^\infty([-M,M],\R^d)$ flux. 
 Assume  $\F$ satisfying  Definition \ref{defgnf}, 
  then 
  $\F$ satisfies Definition \ref{defEE} but 
 the converse is wrong. 
\end{proposition}

\begin{proof} 
 Assume  $\xi \cdot \F''=0$ on a open  sub-interval  $I$. Let $u$ belong in $I$. 
Hence $\xi \cdot \F^{k}(u)=0$   for all $k\geq 2$. But $\F$ satisfies  Definition \ref{defgnf}.  It follows that $\xi =0$.

Conversely take a flux $\F$ such that  
$\F''(u) = \exp(-1/u^2)(1,u,\cdots,u^{d-1})$ . 
Obviously $\F$ satisfies Definition \ref{defEE}.
But $\F$  does not  satisfies Definition \ref{defgnf}
since $d_\F[0]=+\infty$. 
\end{proof}

In  the same way, if   $\F$ satisfies  Definition \ref{defLPTg}  then 
  $\F$ satisfies Definition \ref{defEE}.

For analytic flux, the situation is simpler. 

\begin{proposition}[Analytic nonlinear flux]
Assume  the flux to be  an  analytic function. 
 All previous Definitions \ref{defLPT}, \ref{defgnf}, \ref{defLPTg}, \ref{defEE} are equivalent.
 
\end{proposition}

\begin{proof} 
Again we use Definition \ref{defgnf}. There are two cases. 
\begin{enumerate}
\item If $\F$ is nonlinear  for Definition \ref{defgnf}. 
   By Theorem \ref{thsgo}, Propositions \ref{31LPTg} and \ref{31EE}, 
   $\F$ is nonlinear for other definitions.
\item If $\F$ is not  nonlinear for Definition \ref{defgnf}.
      By Theorem \ref{thsgo}, $\F$ does not satisfy Definition \ref{defLPT}.
 
   Let $u$ be fixed. There exists an hyperplane $H$ such that all derivatives 
 $\F^{(k)}(u) \in H$ for all $k\geq 2$, i.e. there exists $\xi \neq 0$ 
 such that $\xi \cdot \F^{(k)}(u) =0$ for all $k\geq 2$. Using the power series
 expansion of $\F''$ 
near $u$ we see that $\F''$ stays in $H$ near $u$. 
And by the unique analytic extension 
of  $\F''$, $\F''$ stays always in $H$, i.e. $\xi \cdot \F''=0$ everywhere. 
 Thus $\F$ does not satisfies Definition \ref{defEE}.

Integrating the relation  $\xi \cdot F''=0$ we have $\tau + \xi \cdot \F'=0$
for some contant $\tau$. Dividing the relation by $\sqrt{\tau^2+ |\xi|^2}$
 we can assume that $\tau^2+ |\xi|^2=1$. Hence 
$\F$ does not satisfies Definition \ref{defLPTg}.
\\
We incidentally check that Definition  \ref{defLPTg} 
implies Definition \ref{defEE}.
  
\end{enumerate}
\end{proof}

For less smooth flux  we refer to the works of E. Yu. Panov (\cite{Pa95,Pa10}). 

\section{Sobolev estimates \label{sSE} }
\smallskip

In this section, uniform  and optimal Sobolev exponents of  
the family of  highly oscillating solutions
 from Theorem \ref{Tpgamma} are investigated.

\begin{thm}
{\bf [
      Sobolev exponent for highly oscillating solutions]}
\label{Propse}
\alali
 Let $\ue$  be the  $C^1([0,T_0]\times \R^d)$  oscillating solutions given in Theorem \ref{Tpgamma}.
\\
 For all $  1 \leq p < + \infty$, 
the family   $(\ue)_{0 < \eps \leq 1}$ is  uniformly bounded  in  
\bqre  
C^0([0,T_0],W^{s,p}_{loc}(\R^d,\R))
 \quad  \cap \quad  W^{s,p}_{loc}([0,T_0] \times \R^d,\R)    &
  \mbox{
 with }\dis s= \frac{1}{\gamma}.
\eqre
Furthermore, if  $U_0$  is   a {\rm non constant} function, 
then for  all $s > 1/\gamma$ 
the sequence $(\ue)_{0 < \eps <1}$ is unbounded in 
 $ C^0([0,T_0],W_{loc}^{s,p}(\R^d,\R))$  and in  $ W^{s,p}_{loc}([0,T_0]\times\R^d,\R) $.
\end{thm} 
 The Theorem means that the Sobolev exponent $\dis s= \frac{1}{\gamma}$ is {\bf optimal}.  
 It is easily  seen that the sequence $(\ue)_{0 < \eps}$ is uniformly bounded in  $W^{1/\gamma,p}_{loc}$ by interpolation (see remark  \ref{rinter} below). 
The difficult part of the theorem is the optimality. That is to say the sequence is unbounded for too large $s$. 
For this purpose we need to get lower bound of Sobolev norms. Unfortunately, interpolation theory only gives upper bounds.
Thus we use the intrinsic norm. It is rather elementary but quite long to achieve such lower bounds.   
All this section is essentially devoted to compute these  lower bounds to highlight  the conjecture about the maximal smoothing effect in the next section.

Indeed, it is proved below that $\ue$ 
has  order of  $\eps^{1 -s \g}$ in $W^{s,p}_{loc}$ for any $s \in [0,1[$.
\\
The case $p=1$ is the most important, since $L^1$ norm plays an important 
 role for conservation laws. The Sobolev estimates  
 of the initial data are propagated by the semi-group $\mathcal{S}_t$,
 (see \cite{LPT} for $p=1$ and also \cite{PW} for $TV(|\ue - \underline{u}|^s)$). 
 A key point is there is no improvement of the Sobolev exponent 
  of the  family of initial data.

 The  basic idea of the proof is that the  sequence of exact solutions
 $(\ue)_{0<\eps \leq 1}$
 and    the sequence of  approximate oscillating solution given by  
 $\dis \underline{u} + \eps U\left(t,\frac{\phi(t,\x)}{\eps^\g} \right)$
 have similar bounds in Sobolev spaces.


 We  use the  $W^{s,p}$ intrinsic semi-norm instead  the interpolation theory as we explained before. 
 More precisely, following semi-norms    parametrized by
 $Q=Q_d(\x_0,A) = \x_0 + ]-A,A[^d$,   
  where  $A>0$, $\x_0 \in \R^d$,   are used 
to estimate fractional  derivatives 
 in $W^{s,p}_{loc}(\R^d,\R)$ (\cite{Adams}).
\bqre\dis   | V |^p_{\dot{W}^{s,p}(Q_d(\x_0,A))} 
             & =&  
  \dis  \int_{Q_d(\x_0,A))}  \int_{Q_d(\x_0,A))}
\frac{|V(\x)-V(\y)|^p}{|\x - \y|^{d+s p}}d\x d\y.
             \eqre
 Following classical Definitions are used in this section.
\begin{dfni}{\bf [ Estimates in $W^{s,p}_{loc}(\R^d)$]} \label{def31}
\alali
 (i) $u$ is said to be bounded in $W^{s,p}_{loc}(\R^d)$ if  
 \bqre
    \forall \x_0 \in \R^d, \exists A > 0, \exists C \geq 0,&
 \\ 
      \|u\|_{W^{s,p}(Q_d(\x_0,A))} =  & \|u\|_{L^{p}(Q_d(\x_0,A))}+ |u|_{\dot{W}^{s,p}(Q_d(\x_0,A))} \leq  C.
\eqre  
(ii) $\dis (\ue)_{0<\eps \leq 1} $ is said to be bounded in $W^{s,p}_{loc}(\R^d)$ if  
 $$
\forall \x_0 \in \R^d, \exists A > 0, \exists C \geq 0,   
   \forall \eps \in ]0,1], 
      \|\ue \|_{W^{s,p}(Q_d(\x_0,A))}  \leq  C.
$$  
(iii)  Let $\beta \geq 0$, 
   $\dis (\ue)_{0<\eps \leq 1} $ has order  of $\eps^{-\beta}$ in $W^{s,p}_{loc}(\R^d)$,
     denoted by 
  \bqre  & \label{order} \ue  \simeq  \eps^{-\beta}, \\ 
      \mbox{if } &  
                     \forall \x_0 \in \R^d, \exists A > 0, \exists C\geq 1, 
                   \exists \eps_0 \in ]0,1], \forall \eps \in ]0,\eps_0], \\    
       & C^{-1} \; \eps^{-\beta} \leq   \|\ue \|_{W^{s,p}(Q_d(\x_0,A))} 
       \leq  C \; \eps^{-\beta}.
    \eqre 
\end{dfni}
As usual  if  $u$ is bounded in $W^{s,p}_{loc}(\R^d)$  then 
 for any cube  $Q$, $u$ belongs to
 $ W^{s,p}(Q)$.
By  the same way  if $ \ue \simeq  \eps^{-\beta}$ in $W^{s,p}_{loc}(\R^d)$ then
for any cube $Q$ there exists a constant $C \geq 1$ 
and  $ \eps_0 \in ]0,1] $ such that for all $0 < \eps \leq \eps_0$,
$ \dis 
 C^{-1} \; \eps^{-\beta} \leq   \|\ue \|_{W^{s,p}(Q} 
       \leq  C \; \eps^{-\beta}.
$  
\medskip \\

Since solutions of $\eqref{1.1}$ are bounded in $L^\infty$, 
the key point is to focus on fractional derivatives. 
For convenience  $|\x| = |x_1| + \cdots + |x_d|$ 
 and  semi-norms 
 \bqre\dis   | V |^p_{\dot{\widetilde{W}}^{s,p}(Q_d(\x_0,A))} 
             & =&   \dis  
              \int_{Q_d(0,A)}  \int_{Q_d(\x_0,A)} \frac{|V(\x+\h)-V(\x)|^p}{|\h|^{d+sp}}d\x d\h, 
\eqre  
%
are also  used.
 Notice that  
 \bqre | V |_{\dot{W}^{s,p}(Q_d(\x_0,A/2))}  \leq 
         & | V |_{\dot{\widetilde{W}}^{s,p}(Q_d(\x_0,A))}  
         &\leq| V |_{\dot{W}^{s,p}(Q_d(\x_0,2A))} .
 \eqre
 Furthermore,  
  $   | V |_{\dot{\widetilde{W}}^{s,p}(Q_1(\x_0,A))} = | V |_{\dot{W}^{s,p}(Q_1(\x_0,A))} $
 when $V$ is periodic with period 
 $A$ (or $A/2$). 
 Thus, these semi-norms can be useful to estimate bounds in $W^{s,1}_{loc}$.
\\ 


The simplest example of  high frequency oscillating functions     with optimal estimates in Sobolev spaces is investigated in the following Lemma. 
The remainder of the section is devoted to get the same estimates for the the family of  highly oscillating solutions
 from Theorem \ref{Tpgamma}. 
\begin{lem}{\bf [Highly oscillating periodic function on $\R$]} \label{Losc1D}
\alali 
  Let $v$ belong to  $W^{s,p}_{loc}(\R, \R)$, 
$ \gamma > 0$, 
  and for all $0 < \eps \leq 1 $,
\bqre \Ve(\theta) & = &  v(\eps^{-\gamma}\theta).\eqre  
If  $v(.)$ is  a {\rm non constant } periodic function 
  then 
  \bqre \Ve   & \simeq  & \eps^{- s   \gamma}  
\qquad  \mbox{  in } W^{s,p}_{loc}(\R) .
  \eqre
Furthermore, if   $ \Ve(\theta)  =   \ve(\eps^{-\gamma}\theta) $, $\ve$ is one periodic, 
 and  $\ve \to v $ in $C^1$  
 then  $\Ve \simeq   \eps^{- s   \gamma}  
 \mbox{  in } W^{s,p}_{loc}(\R) . $
\end{lem}
\noindent
Notice that the magnitude of $\Ve$ in $W^{s,p}_{loc}$ is independent of $p$.\\
Notice also that  if $\ve \to v $ in $W^{s,p}_{loc}$  
 then  $\ve(\eps^{-\gamma}\theta) \simeq   \eps^{- s   \gamma}  
 \mbox{  in } W^{s,p}_{loc}(\R) . $
 
\medskip
\bpro  In all the sequel one sets $x_0=0$  in  Definition \ref{def31}
 since computations are invariant under translation. 
 \\ 
     First the $L^1_{loc}$ norm is easily bounded  in \cite{CJR}. 
  Let  $A >  1/2$, $X=\eps^{-\gamma}x$, $\Be=\eps^{-\gamma}A$,  
 $\Ne$ the integer such that $ \Ne \leq 2\Be < \Ne+1$ so 
 $2A-1\leq 2A -\eps^\g \leq \eps^{\gamma}\Ne \leq 2A$.   
\bqre 
       \| \Ve \|^p_{L^p([-A,A])} & =& 
         \dis  \int_{-A}^A | \Ve(x)|^p dx   
         =  \dis \eps^{-\gamma} \int_{-\Be}^{\Be} | v(X)|^p dX  \\
               & = &  
 \dis \eps^{-\gamma} \left(  \sum_{k=1}^{\Ne}\int_{-\Be + k-1}^{-\Be +k} | v(X)|^p dX   
        + \int_{-\Be + \Ne}^{\Be} | v(X)|^p dX  \right) \\
  & = &  
 \dis \eps^{-\gamma} \Ne \int_{0}^{1} | v(X)|^p dX   
        +  \eps^{-\gamma} \int_{-\Be + \Ne}^{\Be} | v(X)|^p dX  . \\
\eqre
Finally one has
\bqr  \label{lpel}
          \| \Ve \|_{L^p([-A,A])} 
              & \leq  &  
         (2A+1)^{1/p} \|v\|_{L^p([0,1])},  \\
      \label{lpeg}
  \| \Ve \|_{L^p([-A,A])} 
              & \geq  &  (2A-1)^{1/p} \|v\|_{L^p([0,1])}
            \\
  \nonumber 
 \| \Ve \|_{L^p([-A,A])}  
    & \sim   &
         (2A)^{1/p} \|v\|_{L^p([0,1])} \quad \mbox{ when } \eps \to 0.
\eqr
  $\dis   | \Ve |_{\dot{\widetilde{W}}^{s,p}([-A,A])}  $ is computed 
with the same notations and   $H=\eps^{-\gamma} h$,
\bqre\dis   | \Ve |^p_{\dot{\widetilde{W}}^{s,p}([-A,A])}  & =&  
           \dis  \eps^{(1-sp)\gamma} 
              \int_{-\Be}^{\Be}  \int_{-\Be}^{\Be} \frac{|v(X+H)-v(X)|^p}{|H|^{1+sp}}dX dH.       \eqre
 Let $Var(.)$ be the one periodic function bounded in $L^\infty$  by 
 $ 2^{p} \|v\|^p_{L^p([0,1])}$,
 \bqre 
 \ Var(H) & =& \dis \int_0^1|v(X+H)-v(X)|^p dX.
\eqre
Notice that $Var \equiv 0$ if and only if  $v$ is constant  a.e.

   Using one periodicity of $v$ with respect to $X$ yields 
 as in \eqref{lpel} 
\bqre\dis   | \Ve |^p_{\dot{\widetilde{W}}^{s,p}([-A,A])}  & =&  
           \dis  \eps^{- sp\gamma}
              \int_{-\Be}^{\Be} \left( 
     \eps^{\gamma}  \int_{-\Be}^{\Be} |v(X+H)-v(X)|^pdX
                              \right) \frac{dH}{|H|^{1+sp}},\\
         & \leq &  \dis  \eps^{- sp\gamma}
              \int_{-\Be}^{\Be} \left( 
                (2A+1) Var(H)
                              \right) \frac{dH}{|H|^{1+sp}}
             \leq  \eps^{-sp\gamma}   (2A+1)D^p_\infty,\\
   D^p_B=(D_B)^p &=&   \int_{-B}^{+B}   Var(H)  \frac{dH}{|H|^{1+sp}}. 
    \eqre
Notice that $D_B$ is a true constant related
 to the fractional derivative of $v$
 since for $B=1/2$,
  $\dis D_{1/2} = | v |_{\dot{\widetilde{W}}^{s,p}([-1/2,1/2])}$ and 
 for $B=\infty$ the integral converges.
\\
 The lower bound is obtained by  the same way and finally one has
\bqre 
 \dis   | \Ve |_{\dot{\widetilde{W}}^{s,p}([-A,A])}  & \leq & 
                \eps^{-s\gamma} (2A+1)^{1/p}D_\infty , \\ 
        | \Ve |_{\dot{\widetilde{W}}^{s,p}([-A,A])}  & \geq & 
                \eps^{-s\gamma} (2A-1)^{1/p}D_1 , \\
   | \Ve |_{\dot{\widetilde{W}}^{s,p}([-A,A])}  & \sim & 
                \eps^{-s\gamma} (2A)^{1/p}D_\infty.
\eqre
 Notice also that  $D_B>0$ for $B> 1/2$.
 Otherwise $D_B=0$  implies $Var \equiv  0$ a.e.
 which implies $v$ is  a constant function on $[x_0-2B,x_0+2B]$
 and on $\R$ by periodicity.
\\
A key point in this paper is the lower bound to get sharp estimates. 
Since $D_B$ is non decreasing with respect to $B$,  the previous lower bound 
of $\Ve$ in $W^{s,p}$ implies the following lower bound 
 \bqre
 | \Ve |_{\dot{\widetilde{W}}^{s,p}([-A,A])}  & \geq & 
                \eps^{-s\gamma} (2A-1)^{1/p} | v |_{\dot{\widetilde{W}}^{s,p}([-1/2,1/2])} 
.
\eqre
 With more work,  similar estimates are still valid for 
$ | \Ve |_{\dot{W}^{s,1}([-A,A])}$, see lemmas in \cite{CJR} about triangular changes of variables for oscillatory integrals. But it is enough for our purpose. 
\\
Same computations when $v$ is replaced by $\ve$ are still valid 
which conclude the proof.
\epro

The following lemma is useful to check that  $W^{s,1}$ semi-norms of 
$V: \R \mapsto \R$  and  $W:  \R^d \mapsto \R$ where 
$W(x_1,\cdots,x_d) = V(x_1)$  have the same order.
\begin{lem} \label{lemd1}
 Let $d \geq 2 $, $s>0 $, $A >0$,  $ h_1 >0$,  
  \bqr   
    \mu_{d,s}(h_1) = \dis 
               \int_{0}^{A} \cdots  \int_{0}^{A}
         \frac{h_1^{1+s}}{(h_1 + h_2 + \cdots + h_d)^{d+s}} dh_2 \cdots dh_d,
   \eqr
   then there exists two positive numbers $c_{d,s}, C_{d,s}$
     such that  
 \bqr 
   \label{Idh1}
   0  <  c_{d,s} \leq   \mu_{d,s}(h_1)   \leq  
      C_{d,s} < + \infty,    
 & \forall  A >0, \quad \forall h_1 \in ]0,A]. 
 \eqr 
Inequalities \eqref{Idh1} are still valid for $ h_1 \in ]0,2A]$ 
  with other constants:   \\
$ 0  <  \widetilde{c}_{d,s} \leq   \mu_{d,s}(h_1)   \leq  
      \widetilde{C}_{d,s} < + \infty$.
\end{lem}
The constants $c_{d,s}$ and $C_{d,s} $ are independent of $A >0$. 
Notice that there is a singularity for $\mu_{d,s}$ at $h_1=0$ since 
$\mu_{d,s}(0)=0$ and $\mu_{d,s} > 0$ on $]0,A]$.

\bpro
  It seems that  $\mu_{d,s}(h_1) $ is depending on $A$,
  $\mu_{d,s}(h_1) = \mu_{d,s}^A(h_1)  $. But by homogeneity
  the problem  is reduced to the case  $A=1$ with
 the change of variable $h_i = t_i A$, $0 < t_i < 1$.
\\
 Now $ \mu_{d,s}(t_1)= \mu_{d,s}^1(t_1)= \mu_{d,s}^A(h_1)$ is computed explicitly .
\\
 Let $\mu_{d,s}(h_1,B)$ be 
 $ \dis \int_{0}^{1} \cdots  \int_{0}^{1}
   \frac{t_1^{1+s}}{(t_1 + t_2 + \cdots + t_d + B)^{d+s}} dt_2 \cdots dt_d$
    for $d >1$,  $B \geq 0$. 
Notice that $\mu_{d,s}(t_1)=\mu_{d,s}(t_1,0)$.
\\   For $d=1$, set 
   $\mu_1(t_1,B) = \dis  \frac{t_1^{1+s}}{(t_1  + B)^{1+s}} $, 
      $\mu_1(t_1)= \ \mu_1(t_1,0) = 1 $.  
  The identity  
\bqre 
\dis \int_0^1 \frac{dt}{(t+B)^{(1+j+s)}}  & =& 
        (j+s)^{-1}\left (B^{-(j+s)} - (B+1)^{-(j+s)}  \right),
  \eqre
yields
$ 
 (j+s) \mu_{1+j}(t_1,B)    =  \mu_j(t_1,B) - \mu_j(t_1,B+1),  
$ 
and proceeding by induction with the notations
  $\dis   \gamma_{d,s}=  \frac{1}{(d-1 + s)\cdots (1+s)}$, 
 $\dis C_n^k = \frac{n!}{k! (n-k)!} $, 
\bqre
 \mu_{d,s}(t_1,B)  & =& 
     \dis \gamma_{d,s} 
\sum_{k=0}^{d-1} C_{d-1}^k (-1)^k \mu_1(t_1,B+ k ). 
\eqre 
Hence, for $B=0$,  
\bqre
 \mu_{d,s}(t_1)  &= &
     \dis  \gamma_{d,s}\sum_{k=0}^{d-1} C_{d-1}^k (-1)^k 
                                    \frac{t_1^{1+s}}{(t_1+k)^{1+s}}, 
\eqre
which gives $\mu_{d,s}(0+)= \gamma_{d,s} > 0 $.
Now, $\mu_{d,s}(.)$ belongs in $C^0(]0,1],\R^+)$, $\mu_{d,s}(.)$ is positive on $]0,1]$
 with a  positive right limit at $t_1=0$, thus positive constants stated 
in the lemma exist.  
\\
For instance when $d=2$, $C_{2,s}$ is $\gamma_{2,s}=1/(1+s)$ and 
$c_{2,s}= (1 -2^{-(1+s)})/(1+s)$, since $\mu_2$ is decreasing.
\\ 
Notice that $C_{d,s} \leq \gamma_{d,s}$  for all $d \geq 2$.
It suffices  to proceed by induction with  this inequality
$ 
\dis \int_0^1 \frac{dt}{(t+B)^{(1+j+s)}}  \leq  
        (j+s)^{-1}B^{-(j+s)}
 .$ 
But $\gamma_{d,s}$ is the right limit of $\mu_{d,s}$ at $t_1 = 0$. 
Then $C_{d,s} = \gamma_{d,s}$ which  concludes the proof for 
 $h_1 \in ]0,A]$. 
On $]0,2A]$ it suffices to take 
$0 < \widetilde{c}_{d,s}=\inf_{]0,2]} \mu_{d,s}$ 
 and $+ \infty > \widetilde{C}_{d,s}=\sup_{]0,2]} \mu_{d,s}$.
 \epro
\medskip \\    
Our examples of oscillating solutions  is related to the following key example. 
For instance $\Ve$  defined by $\ue = \eps \;  \Ve $  where $\ue$ is  the solution of 
\\
$ \partial_t (\ue) + \partial_x | \ue|^{1+\gamma}=0,$ $  \ue(0,x)= 0 + \eps \;  U(0,\eps^{-\gamma} x)$ 
satisfies the assumption of the lemma on a bounded time interval
(\cite{CJ1}).
\begin{lem}{\bf [Example of highly periodic oscillations on 
                  $[0,T] \times \R$]}
 \label{Losc21}
\alali
Let $T$, $\g$ be positive.
If $U$ belongs to $C^1([0,T]\times \R/\Z,\R)$ 
and non constant, 
then 
$\Ve(t,x)=  U(t,\eps^{-\gamma} x) \simeq \eps^{-s\g}$ in 
  $C^0([0,T], W^{s,p}_{loc}(\R))
 \cap   W^{s,p}_{loc}(]0,T[ \times \R) $.
\end{lem}
\begin{remark}\label{rinter}
Notice that the upper bound is quite easy to get.
It directly follows from the fact that $W^{s,p}$ is an interpolated 
 space of exponent $\theta=s$ between $L^p=W^{0,p}$ and $W^{1,p}$, \cite{TartarI}. 
 But we  also want  a lower bound to obtain an optimal estimate. 
 This is a very crucial point in our study.
 For this purpose we use the intrinsic semi-norm in the  proofs. 
 The computations are elementary but long.

 The same remark is still valid for all the next lemmas in this section. 
 \end{remark}
 
\bpro 
         First  the fractional derivative w.r.t. $\x$ is estimated. 
        Second the whole fractional derivative in $(t,\x)$ is obtained.
  \medskip  
  \\
\underline{Bounds in $L^\infty([0,T], W^{s,p}_{loc}(\R)) $}:
 There exists $t_0 \in ]0,T[$ such that $\theta \mapsto U(t_0,\theta)$
 is non constant 
since $U $ is non constant and continuous on $[0,T]\times\R/\Z$.
 For  $t_0$ fixed the sharp estimate is a consequence of Lemma \ref{Losc1D}.
 For another  $t$, we get the same order $\eps^{-s\gamma}$ or $\eps^{0}=1$. 
 Finally, constants involving in this estimate depend continuously of $t$
 so the   bound in $L^\infty([0,T], W^{s,p}_{loc}(\R)) $ is obtained.
 Since $U \in C^1$ this bound is automatically in  $C^0([0,T], W^{s,p}_{loc}(\R))$.
 \medskip 
  \\
\underline{Bounds in $ W^{s,p}_{loc}(]0,T[ \times \R)) $}:
The only problem is to  estimate for $x_0 \in \R$, 
$t_0 \in ]0,T[$ and 
$\min(t_0,T-t_0)> A>0$, the quadruple integral 
\bqre 
 \dis 
 IA
   & = &\dis |\Ve|^p_{\dot{\widetilde{W^{s,p}}}( [t_0 -A,t_0+A]\times [x_0-A,x_0+A])  } 
 \\ & = & 
   \int_{t_0-A}^{t_0+A} \int_{x_0-A}^{x_0+A} \int_{-A}^{A} \int_{-A}^{A}
  \frac{|U(t+\tau,\eps^{-\gamma} (x+\xi))-U(t,\eps^{-\gamma} x)|^p}{(|\tau|+|\xi|)^{2+sp}} d\xi d\tau dx dt. 
\eqre
{\it   Upper bound of $IA$:} \\
 Let $Num$ be the numerator of the previous fraction,
  $Q$ be $\dis U(t,\eps^{-\gamma} (x+\xi))-U(t,\eps^{-\gamma} x)$, 
$R$ be $\dis U(t+\tau,\eps^{-\gamma} (x+\xi))-U(t,\eps^{-\gamma}( x+\xi))$
then 
$ Num = |Q+R|^p \leq 2^{p-1} (|Q|^p+|R|^p)$.
\\
Previous inequality  implies $ IA \leq 2^{p-1} (IQ+ IR) $ with obvious notations.
\bqre 
 IQ  & =&  \int_{}^{} \int_{}^{} \int_{}^{} \int_{}^{}
  \frac{|U(t,\eps^{-\gamma} (x+\xi))-U(t,\eps^{-\gamma} x)|^p}{(|\tau|+|\xi|)^{2+sp}} d\xi d\tau dx dt, \\ 
  & =&  \int_{}^{}  \int_{}^{} \int_{}^{}
  \frac{|U(t,\eps^{-\gamma} (x+\xi))-U(t,\eps^{-\gamma} x)|^p}{|\xi|^{1+sp}}
       \mu_{2,sp}(\xi) d\xi dx dt,
\eqre
 with $\mu_{2,sp}(.)$  is defined in Lemma \ref{lemd1}.
 Using Lemmas \ref{Losc1D}, \ref{lemd1} yield  $IQ \simeq \eps^{-s\g}$.
\\
$IR$ is easily bounded since 
\bqre 
 IR  & =&  \int_{}^{} \int_{}^{} \int_{}^{} \int_{}^{}
  \frac{|U(t+\tau,\eps^{-\gamma} (x+\xi))-U(t,\eps^{-\gamma} (x+\xi))|^p}{(|\tau|+|\xi|)^{2+sp}} d\xi d\tau dx dt, \\ 
  & \leq &  \int_{}^{} \int_{}^{}  \ \int_{}^{} \int_{}^{}
  \frac{\|\partial_t U\|^p_{L^\infty} |\tau|^p}{(|\tau|+|\xi|)^{2+sp}}
       d\tau d\xi dx dt
\\ & \leq &  8 A^2\|\partial_t U\|^p_{L^\infty}
     \int_0^A |\tau|^{p(1-s) -1 } \mu_{2,sp}(\tau) d\tau,
\eqre
which is finite, 
so $IA \leq IQ+IR = \mathcal{O}(\eps^{-s p \g})$. 
\medskip \\
{\it Lower bound  of $IA$:} \\
 We again use notations $Q$, $R$, $Num$.
  By a convex inequality, the numerator satisfies:   
$Num= \dis |Q+R|^p  \geq |Q|^p - p |Q|^{p-1} |R| = |Q|^p - O( |\tau| |Q|^{p-1}) $  
 since $R = O(\tau)$.   
 Then  $IA \geq IQ -O( IS) $, 
where  $IQ$ has order  of $\eps^{-s p \g}$.
The term $IS$ has a lower order  as we can  find
 after the following  similar computations 
as in the proof of Lemma \ref{Losc1D}.
Notice first  that  for all positive  numbers $A$, $b$, 
 $\dis \int_0^A \frac{\tau}{(\tau +b )^{2+\beta}}d\tau \leq \frac{C}{2 b^\beta}$
 where $\beta >0$  and  
 $C=2 \int_0^{+\infty} \frac{\tau}{(\tau +1 )^{2+\beta}}d\tau < + \infty$.
 Now integrating on $\tau$ yields
\bqre 
 IS  & =&  \int_{}^{} \int_{}^{} \int_{}^{} \int_{}^{}
  \frac{ |\tau| |Q|^{p-1}}{(|\tau|+|\xi|)^{2+sp}} d\xi d\tau dx dt
     \leq  C \int_{}^{} \int_{}^{} \int_{}^{}
  \frac{  |Q|^{p-1}}{|\xi|^{sp}} d\xi  dx dt.
\eqre
We set $\eta = \eps^\g$, $X=x/\eta$, $\Xi=\xi/\eta$, the previous 
inequality becomes
\bqre 
 IS  & \leq &
       C \eta^{2-sp} \int_{0}^{T} \int_{A/\eta}^{A/\eta} \int_{A/\eta}^{A/\eta} 
  \frac{  |Q|^{p-1}}{|\Xi|^{sp}} d\Xi  dX dt.
\eqre
We now focus on the integral with respect to $\Xi$ and remark
that $Q=O(1)$ and also $Q=O(\Xi)$ since $U$ is $C^1$.
 \bqre 
 \int_{- A/\eta}^{A/\eta}  
  \frac{  |Q|^{p-1}}{|\Xi|^{sp}} d\Xi    
 &= &   \int_{|\Xi| <1}^{}  
  \frac{  |Q|^{p-1}}{|\Xi|^{sp}} d\Xi  + 
   \int_{1 <|\Xi| <A/\eta}^{}  
  \frac{  |Q|^{p-1}}{|\Xi|^{sp}} d\Xi \\
& \leq &
    \int_{|\Xi| <1}^{}  
  \frac{ O( |\Xi|^{p-1})}{|\Xi|^{sp}} d\Xi  + 
   \int_{1 <|\Xi| <A/\eta}^{}  
  \frac{  O(1)}{|\Xi|^{sp}} d\Xi    \\
 & \leq &
    \int_{|\Xi| <1}^{}  
   O( |\Xi|^{p(1-s)-1}) d\Xi  + 
    O( g(\eta))  =  O(1)+  O( g(\eta)),
 \eqre
 where $g(\eta)=  \eta^{sp-1}$ if $sp \neq 1$,else $ g(\eta)= \ln (\eta)$.
\\
To bound $IS$, we notice that the integral $\eta\int_{- A/\eta}^{A/\eta} dX $ 
is bounded by periodicity and we can take the supremum with respect $t$
 on $[0,T]$.
So $IS = O(1)$ if $sp\neq 1$ else $IS= O(\ln(\eta))$ 
which is enough to  have a lower order than $IQ$. 
\medskip \\
In conclusion, the bounds of $IA$ 
  yield 
  $ \Ve \simeq \eps^{-s  \g}$
 in  $W^{s,p}_{loc}([0,T] \times \R)$.
\epro

Now, we estimates  the Sobolev norm for  the multidimensional case with one phase.
\begin{lem}{\bf [Example of highly periodic oscillations on $\R^d$]}
 \label{LoscD1}
\alali
  Let $v$ belong to   $W^{s,p}_{loc}(\R, \R)$, 
$ \gamma > 0$, $\psi(\x) = \vv \cdot \x  + b $ where $\vv \in \R^d$, $b \in \R$ 
  and  $0 < \eps < 1 $, 
 \bqre \We(\x) & = &  \dis  v(\eps^{-\gamma}\psi(\x)).
 \eqre
If  $v$ is a {\rm non constant } periodic function and $\nabla \psi \neq 0$,
 then 
 \bqre 
    \We  & \simeq &   \eps^{- s \gamma}  \quad \mbox{ in } W^{s,p}_{loc}(\R^d,\R).
 \eqre
Furthermore,  when  functions $\ve$ are  one periodic function for all $\eps \in ]0,1]$,
  which converge towards $v$ in $C^1$ and 
 $ \We(\x)  =  \dis  \ve(\eps^{-\gamma}\psi(\x)) $, 
the conclusion holds true.
\end{lem}

\bpro
 The expounded proof has three steps.  
Let $M$ be a $d\times d$ 
 non-degenerate matrix  and $B \in \R^d$
such that   $X_1=\psi(\x)$ where  $X=(X_1,\cdots,X_d)=M\x + B$. 
 $M$ exists since $\vv \neq 0$.  
\\

\underline{Step 1}: When  $W(\x)=U(M\x +b)$ since  $\det M \neq 0$,
 $W$ and $U$ are the same order in $W^{s,p}_{loc}$. 
 More precisely, fix following positive constants 
$m_0=|\det M|>0$, 
 $m_{1}=\||M\|| =\sup\{|M\x|,\, |\x|=1 \}>0$, 
 $m_{-1}=\||M^{-1}\|| >0$, $0 < r < R$ such that  
 $Q_d(X_0,r) \subset M Q_d(\x_0,1) \subset Q_d(X_0,R)$ where $X_0=M\x_0 +B$.
 Performing the change of variables $X=M\x + B$, $Y=M\y + B$ yields
 for any $\x_0 \in \R^d$ and any $A >0$
\bqre  m_0^{-1} \| U \|_{L^{p}(Q_d(X_0,rA))} 
        & \leq   \| W \|_{L^{p}(Q_d(\x_0,A))}  \leq   
         &  m_0^{-1} \| U \|_{L^{p}(Q_d(X_0,RA))}  , \\
\frac{ m_0^{-2}}{m_{-1}^{(d+sp)}}    | U |_{\dot{W}^{s,p}(Q_d(X_0,rA))} 
    & \leq  | W |_{\dot{W}^{s,p}(Q_d(\x_0,A))}  \leq  &
  \frac{m_0^{-2}}{m_{1}^{-(d+sp)}}  | U |_{\dot{W}^{s,p}(Q_d(X_0,RA))}  .
\eqre
\underline{Step 2}: Assume $\psi(\x)= x_1$, 
  i.e. $W(\x)=W(x_1,\cdots,x_d)=w(x_1)$, 
 $x_0=\psi(\x_0)$, then   $W$ in $W^{s,p}_{loc}(\R^d)$
     and $w$ in $W^{s,p}_{loc}(\R)$ have the same order. 
  More precisely, elementary computations yield
 \bqre   
            \| W \|_{L^{1}(Q_d(\x_0,A))}   & =  &  
             (2A)^{d-1} \| w \|_{L^{1}(Q_1(x_0,A))}  , \\ 
      | W |_{\dot{\widetilde{W}}^{s,p}(Q_d(\x_0,A))}  &\leq  &
           (2A)^{d-1} C_{d,sp} | U |_{\dot{\widetilde{W}}^{s,p}(Q_1(x_0,A))} \\
       & \geq & (2A)^{d-1} c_{d,sp} | U |_{\dot{\widetilde{W}}^{s,p}(Q_1(x_0,A))} .
\eqre
The two last inequalities and constants come from Lemma \ref{lemd1} since
\bqre 
| W |_{\dot{\widetilde{W}}^{s,p}(Q_d(\x_0,A))}  & =& 
  \int_{Q_d(0,A)}\int_{Q_d(\x_0,A)} \frac{|w(x_1+h_1)-w(x_1)|}{|\h|^{d+sp}} d\x d\h \\
 & =& 
 (2A)^{d-1} \int_{-A}^A\int_{x_0-A}^{x_0+A} \frac{|w(x_1+h_1)-w(x_1)|}{|h_1|^{1+sp}} 
       \mu_{d,sp}(h_1)dx_1 dh_1.
\eqre
 
\underline{Step 3}:  By step 1,  
  $\We(\x)=\Ve(\eps^{-\gamma}\psi(\x)) \simeq 
   \Ve(\eps^{-\gamma}\x_1)$ in $W^{s,p}_{loc}(\R^d)$,
 by  step 2,   $ \x \mapsto \Ve(\eps^{-\gamma}\x_1)$ 
       and   $ x_1 \mapsto \Ve(\eps^{-\gamma}\x_1)$ 
 have   the same 
 order in $W^{s,p}_{loc}(\R^d)$ and $W^{s,p}_{loc}(\R)$.
Finally  we have by   Lemma \ref{Losc1D} $\We \simeq \eps^{-s\gamma}$
 in $W^{s,p}_{loc}(\R^d)$.
\epro

It is the last step
to estimate the Sobolev norm for the multidimensional case  before proving Theorem \ref{Propse}.
\begin{lem}{\bf [Example of highly periodic oscillations on $[0,T]\times\R^d$]}
 \label{LoscD2}
\alali
  Let $U$ belong to   $W^{s,p}_{loc}(\R, \R)$, 
$ \gamma > 0$, $\varphi(t,\x) = \vv \cdot \x  + b \, t $ 
where $\vv \in \R^d$, $b \in \R$ 
  and  $0 < \eps < 1 $, 
 \bqre \We(\x) & = &  \dis  U(t,\eps^{-\gamma}\varphi(t,\x)).
 \eqre
If  $U$ is a {\rm non constant } function  in 
 $C^1([0,T]\times\R/\Z,\R)$ and $\vv \neq 0_{\R^d}$,
 then 
 \bqre 
    \We  & \simeq &   \eps^{- s \gamma}  
\quad \mbox{ in } W^{s,p}_{loc}([0,T]]\times\R^d,\R).
 \eqre
Furthermore,  when  $\Ue$ belongs to $C^1([0,T]\times\R/\Z,\R)$ for all $\eps \in ]0,1]$
 converging towards $U$ in $C^1$ and 
 $ \We(\x)  =  \dis  \Ue(t,\eps^{-\gamma}\varphi(t,\x)) $, 
the conclusion holds true.
\end{lem}

\bpro
We proceed as  in the previous proofs. 
First with a linear change of variable 
$(t,\x) \mapsto (t,\y)$ with $\y_1=\varphi(t,\x)$. 
$\We$ has the same estimates than 
 $\Ve= U(t,\eps^{-\gamma}y_1) $ in $W^{s,p}_{loc}(]0,T[\times\R^d,\R)$. 
Notice that the change of variable depends on $t$
varying in the compact set  $[0,T]$. So we have uniform estimates of positive constants $m_0$, $m_1$, $m_{-1}$ used in  the proof of Lemma \ref{LoscD1}.
\\
Now, the estimates of $\Ve$ in  $W^{s,p}_{loc}(]0,T[\times\R^d,\R)$
and in $W^{s,p}_{loc}(]0,T[\times\R,\R)$ have the same order since
\bqre 
& \dis  \int_{-A}^A \cdots \int_{-A}^A 
\frac{dh_0 dh_1\cdots dh_d }{(|h_0| + |h_1|+ \cdots |h_d|)^{1+d+sp}}
\\= &\dis  \int_{-A}^A \cdots \int_{-A}^A 
  \frac{dh_0 dh_1 }{(|h_0| + |h_1|)^{2+sp}}
\frac{(|h_0|+|h_1|)^{1+(sp+1)} dh_2\cdots dh_d }
     {(|h_0| + |h_1|+ \cdots |h_d|)^{d+(sp+1)}}
 \\= &\dis  \int_{-A}^A  \int_{-A}^A 
  \frac{dh_0 dh_1 }{(|h_0| + |h_1|)^{2+sp}}
     \mu_{2,(sp+1)} (|h_0|+|h_1|) 
\eqre
where $h_0$ plays the role of time. 
From the bounds of $\mu_{2,(sp+1)} (|h_0|+|h_1|) $  on $] 0,2A]$
 see Lemma \ref{lemd1},
 we can conclude with 
 Lemma \ref{Losc21}
\\
With a smooth extension of  $U$ on $[-\delta,T+\delta]\times \R/\Z$, for 
 a small positive $\delta$, we obtain 
 estimates in $W^{s,p}_{loc}([0,T]\times\R^d,\R)$.
\epro

We are now able to prove the Theorem by using Lemma \ref{LoscD1} and the method of characteristics.
\\

{\bf Proof of Theorem \ref{Propse}:}
\underline{Bounds $L^\infty([0,T_0],W^{s,p}_{loc}(\R^d) $}:  
Such bounds give bounds  in  $C^0([0,T_0],W^{s,p}_{loc}) $ since $\ue$ is in $C^1$.

For $t=0$, it is only an application of Lemma \ref{LoscD1}. 
The profile $U(t,.)$ is non constant for each $t$, else $U_0$ must be constant 
 by the method  of characteristics.  
 And the estimates are uniform.
\medskip \\
 \underline{Bounds in  $W^{s,p}_{loc}([0,T_0]\times \R^d) $}
The semi-norms $|.|_{\dot{\widetilde{W}}^{s,p}(Q_{d+1}(\y_0,A))}$, 
where $\y_0=(t_0,\x_0)$,  needs some precautions to use on 
 $[0,T_0]\times \R^d $. $\y_0$ must be such that $0 < t_0 < T_0$
 and $ A < \min(t_0,T_0 - t_0) $. Furthermore, 
only $W^{s,p}_{loc}(]0,T_0[\times \R^d)$ smoothness can be estimate. 
Indeed, $(\ue)_{0<\eps \leq 1} $ is bounded in  $W^{s,p}_{loc}([0,T_0]\times \R^d)$.
To prove this, let us use the following trick.
By the methods of characteristics the family of solutions 
$(\ue)_{0<\eps \leq 1} $ exists on a maximal time interval
   $]-\delta, T_1[ $, with  $0 < \delta < T_0 < T_1 $. 
 Notice that solutions exist for negative time since the initial data is smooth.
Now  estimates in $W^{s,p}_{loc}(]-\delta,T_1[\times \R^d)$ will be obtained
 which is sufficient  to get smoothness in 
 $W^{s,p}_{loc}([0,T_0]\times \R^d)$.
 Now using lemma \ref{LoscD1} completes the proof.
\epro 

\section{Super critical geometric optics} \label{scgo}
\smallskip

Now we can exhibit the supercritical geometric optics and  some implications about 
the maximal smoothing effect for solutions of conservation laws with $L^\infty$ initial data.

\subsection{ Propagation of highest frequency waves }
\alali

In Theorem \ref{Tpgamma} we saw that the frequencies of waves are related 
to an orthogonality condition  between the phase gradient and the derivatives of the flux. 
Theorem \ref{thmal} tell us where the flux reach is maximal degeneracy 
 and which direction the phase gradient has to be chosen. 
Thus we can build  a geometric optics  expansion with the highest frequencies. 
The  uniform Sobolev estimates of such family of oscillating solutions highlight the conjecture about the maximal smoothing effect below.

\begin{thm}{\bf[Bound of the maximal smoothing effect]} 
\label{thsgo}
 \alali
  Let $\F$  be a nonlinear flux which belongs to $C^\infty([-M,M],\R^d)$. 
 Let $\al_{sup}$ be the sharp measurement of the flux non-linearity.
   Then there exist  a constant $\underline{u} \in [-M,M]$, a time $T_0>0$, 
    and a sequence of initial data  
   $(u_0^\eps)_{0<\eps<1}$  such that 
   $ 
      \|u_0^\eps -\underline{u}\|_{L^\infty(\R^d)} < \eps, 
    $ 
 and 
 the sequence of entropy solutions  $(\ue)_{0<\eps<1}$
associated with  conservation law \eqref{1.1}
 satisfying: 
%
\begin{itemize}
\item 
   for  all  
$  s \leq \alpha_{\sup} $, 
the sequence  $(\ue)_{0<\eps<1}$ is uniformly bounded
 \\  in $ 
   W^{s,1}_{loc}([0,T_0]\times \R^d) \cap  C^0([0,T_0],W_{loc}^{s,1}(\R^d)) $,   
 \item 
for all $s>\alpha_{\sup}$,
 the sequence  $(\ue)_{0<\eps<1}$ is  {\bf unbounded}
  \\ in 
  $ W^{s,1}_{loc}([0,T_0]\times \R^d) $  and in $  C^0([0,T_0],W_{loc}^{s,1}(\R^d))$.
\end{itemize}
\end{thm}

\bpro
  The proof is a   consequence of previous theorems.
  By Theorem \ref{thmal}, there exists $\underline{u} \in [-M,M]$ such that 
 $\al =  \frac{1}{d_F[\underline{u}]}$.  Let $U_0$ be a  non constant smooth periodic function 
  such that: 
 $-M \leq \underline{u} + U_0(\theta) \leq M$ for all $\theta$. 
 
 Let $\vv \in \R^d$ such that $\a^{k}(\underline{u}) \centerdot v = 0 $ and $\vv \neq 0$  
for $k=1,\cdots, d_\F[\underline{u}]-1$. Such $\vv$ exists by Definition of $d_\F[\underline{u}]$.

Now, let $(u_\eps)$ be the family of smooth solutions given by Theorem \ref{Tpgamma}.
Theorem \ref{Propse}  is the desired conclusion. 
\epro

\subsection{Highlight of the  Lions, Perthame, Tadmor conjecture} \label{sconj}
\alali

Let us recall the introductory  section \ref{ssse} and use the notations therein.
In \cite{LPT}, 
 the authors obtained  a kinetic formulation of conservation law \eqref{1.1}
 and  used averaging lemmas.  
 With only initial data uniformly bounded in  $L^\infty$, they
  proved  an {\it uniform}  smoothing effect in $W_{loc}^{s,1}$ for all  positive time. 
Thus the best uniform smoothing effect in Sobolev spaces   is at least  (\cite{LPT}): 
\bqre \dis  \frac{\alpha_{\sup}}{2 +\alpha_{\sup}} & \leq &  s_{\sup}.  \eqre 
Theorem \ref{thsgo} gives  an upper  bound for the $W^{s,1}-$regularizing effect  \eqref{goal}:
 \bqre s_{\sup} & \leq  & \al_{\sup}, \eqre 
Indeed,    let us denote  $ \mathcal{B}^\infty(\underline{u},\rho))= 
 \{ u \in L^\infty(\R^d,\R),\; \|u - \underline{u}\|_{L^\infty(\R^d,\R)} < \rho  \}$ and 
$\mathcal{S}_t$ the semi-group associated with  conservation law  \eqref{1.1}.     
Theorem \ref{thsgo}  proves that  
 for  a well chosen   $\underline{u} \in [-M,M]$, there exists $T_0>0$,
 such that for all $\rho >0$ and    for all $ 0 < t < T_0$, 
 $\mathcal{S}_t(\mathcal{B}^\infty(\underline{u},\rho))$ is not a {\it  bounded } subset 
 of $W^{s,1}_{loc}(\R^d_\x)$  for all $s> \al_{\sup}$.

 This result yields some remarks.
\begin{remark} {\it  Optimality for large dimension or large nonlinear degeneracy.}
 \\
By Theorem \ref{thmal},  
$\al_{\sup} \leq \dfrac{1}{d}$ , so   for large dimension $\al_{\sup}$  is small. 
Using a better  lower bound of $s_{\sup}$  from \cite{TT} we have:  
\bqre 
 \dis    \frac{\alpha_{\sup}}{ 1 + 2\; \alpha_{\sup}}  
\leq  s_{\sup}\leq \al_{\sup} \leq \dfrac{1}{d}. 
\eqre 
Thus for  large dimension  ($d >>1$)  or large nonlinear degeneracy ($\alpha_{\sup} <<1$)
we have asymptotically the right $s_{\sup} \sim  \alpha_{\sup} . $ 
\end{remark}

\begin{remark}{\it  In  $W^{s,p}$, $1 < p < + \infty$,} our geometric optics  expansion shows that   $s_{\sup}^p \leq \al_{\sup}$
 by  Theorem \ref{Propse}, where $s_{\sup}^p$ denotes the maximal uniform smoothing effect   in $W^{s,p}$.  
In other words our example is not related to the parameter $p$.   Other examples show  the importance of the  parameter $p$ 
in remark \ref{rind0}. 
\end{remark}

 \begin{remark} \label{rind0} {\it  Critical entropy solutions in the one dimensional case.}      
\\ 
 In \cite{DW,CJ1} special initial data $u_0(x) \in W^{s,p}(\R)$ are built,  
 in \cite{DW}   a piecewise smooth initial data and in \cite{CJ1} a continuous oscillating initial data.
The entropy solution $u(t,x)$ preserves this smoothness  at least on a bounded  time interval, indeed,   before the waves  interactions.
The regularity $s< \alpha_{\sup}$ can be choose  as close as we want to $\alpha_{\sup}$. 
 Furthermore the parameter $p$ is related to the parameter $s$:  $p = \dfrac{1}{s}$. 
Indeed, we cannot expect a greater parameter $p$ since $W^{s,p}(\R) \subset C^0(\R)$ for $p > \dfrac{1}{s}$.  
It would  be interesting to construct such solutions with the almost critical regularity in the multidimensional case.  
\end{remark}
 \begin{remark} \label{rind} {\it  Fractional $BV$ spaces.}      
\\ 
Recently in \cite{CJ1,BGJ6}, for  the one dimensional case and for all   nonlinear  degenerate convex fluxes,
conjecture \eqref{conjLPT} is reached.  
 Furthermore, entropy solutions satisfy  a new one-sided Holder condition. 
 For this purpose,   new spaces are introduced in the framework of conservation laws: 
the fractional $BV$ spaces  $BV^s $.
 $BV^s$ functions have  a structure similar to the one of $BV$ maps,   for all   $s \in ]0,1]$.  
$BV^s$  spaces seems to be  natural spaces to capture
  the regularizing effect for one dimensional scalar conservation laws. 
\end{remark}
This last  promising remark concludes this paper. 
\\

\bigskip

{\bf Acknowledgments.} The author   thanks 
   Florent Berthelin, Yann Brenier, Gui-Qiang Chen, David Chiron, George Comte,  Camillo De Lellis, Pierre-Emmanuel Jabin,  Michel Merle, 
   Beno\^{\i}t Perthame, 
  Michel Rascle, Jeffrey Rauch
   and Eitan Tadmor for fruitful
discussions  related to this subject.
 The author is greatly indebted to anonymous referees for improving the paper.

     
\end{document}